\documentclass{amsart}
\pdfoutput=1

\usepackage{lmodern}
\usepackage[utf8]{inputenc}
\usepackage[mathscr]{eucal}% So-called Euler fonts, allows \mathscr
\usepackage{amssymb}% allows special arrows like >-> e.g.
\usepackage{stmaryrd}
\usepackage[usenames,dvipsnames]{xcolor}
\usepackage{xspace}%allows smart spacing after period
\usepackage{amsthm}% allows theorem* , e.g.
\usepackage{bbold}% allows \unit \mathbb{1}
\usepackage[normalem]{ulem}% allows \sout to cross-out text.
\usepackage{float}
\usepackage{adjustbox}
\usepackage{enumitem}% allows easy change of label
\setlist{leftmargin=*, wide, labelindent=0pt}
\setlist[enumerate]{label*=(\alph*),ref=\alph*}

\usepackage{array}% allows array
\usepackage{amsmath}% allows pmatrix
\usepackage{wasysym}% allows LEFTCIRCLE
\usepackage{etoolbox}% allows for \ifblank
\usepackage{mathdots}% allows diagonal dots

\numberwithin{equation}{section}% makes equat numb contain the section
\usepackage[all]{xy}
\SelectTips{cm}{}
\newdir{ >}{{}*!/-10pt/\dir{>}}

\usepackage{xr-hyper}
\usepackage[colorlinks=true,linkcolor={Black},citecolor={Black},urlcolor={Black}]{hyperref}
\usepackage[nameinlink]{cleveref}% allows references via \cref and \Cref which automatically repeat the name (Thm, Prop, etc) of the ref.

\crefname{Thm}{Theorem}{Theorems}
\crefname{Rem}{Remark}{Remarks}
\crefname{Prop}{Proposition}{Propositions}
\crefname{Cor}{Corollary}{Corollaries}
\crefname{Cons}{Construction}{Constructions}
\crefname{Exa}{Example}{Examples}
\crefname{Lem}{Lemma}{Lemmas}
\crefname{Rec}{Recollection}{Recollections}

\swapnumbers %pour que le numero apparaisse devant le theoreme
\newtheorem{Cor}[equation]{Corollary}
\newtheorem{Lem}[equation]{Lemma}

\newtheorem{Prop}[equation]{Proposition}
\newtheorem{Thm}[equation]{Theorem}

\theoremstyle{remark}
\newtheorem{Def}[equation]{Definition}
\newtheorem{Not}[equation]{Notation}
\newtheorem{Exa}[equation]{Example}

\newtheorem{Hyp}[equation]{Hypothesis}
\newtheorem{Rem}[equation]{Remark}

\newtheorem{Rec}[equation]{Recollection}
\newtheorem{Cons}[equation]{Construction}

\newtheorem*{Ack}{Acknowledgements}

\newtheorem{Conv}[equation]{Convention}

\newcommand{\nc}{\newcommand}
\nc{\dmo}{\DeclareMathOperator}

\dmo{\coker}{coker}
\dmo{\cone}{cone}
\dmo{\Der}{D}% ground notation for derived categories
\dmo{\DAM}{\DAMbig^{\geom}}%
\dmo{\DAMbig}{DAM}%
\dmo{\Ext}{Ext}
\dmo{\Gal}{Gal}
\dmo{\Hm}{H}% (co)homology
\dmo{\Hom}{Hom}
\dmo{\Id}{Id}
\dmo{\Ind}{Ind}
\dmo{\Infl}{Infl}
\dmo{\Ker}{Ker}
\dmo{\Mod}{Mod}% modules
\dmo{\opname}{op}
\dmo{\perm}{perm}% finite permutation modules
\dmo{\Perm}{Perm}% permutation modules
\dmo{\SH}{SH}% ground name for cat of compact spectra
\dmo{\SHmot}{SH^{\mathrm{c}}_{\bbA^{\!1}}}% ground name for cat of compact motivic spectra
\dmo{\Spc}{Spc}
\dmo{\Spec}{Spec}
\dmo{\Spech}{\Spec^{h}}
\dmo{\stab}{stab}% stable category of fin. gen. mod.
\dmo{\Stab}{Stab}% stable category of non-fin. gen. mod.
\dmo{\supp}{supp}
\dmo{\thick}{thick}
\dmo{\Locname}{Loc}% for localizing subcat (w/o tensor)

\nc{\Inj}{\mathrm{Inj}}% injective complexes
\nc{\cat}[1]{\mathscr{#1}}%or: \nc{\cat}[1]{\mathcal{#1}}
\nc{\cK}{\cat{K}}
\nc{\cL}{\cat{L}}
\nc{\colim}{\mathop{\mathrm{colim}}}
\nc{\hocolim}{\mathop{\mathrm{hocolim}}}
\nc{\cP}{\cat{P}}
\nc{\cQ}{\cat{Q}}
\nc{\cS}{\cat{S}}
\nc{\cT}{\cat{T}}
\nc{\CAlg}{\mathrm{CAlg}}
\nc{\eg}{{\sl e.g.}\@\xspace}
\nc{\gp}{\mathfrak{p}}% prime p
\nc{\gq}{\mathfrak{q}}% prime q
\nc{\hook}{\hookrightarrow}
\nc{\ie}{{\sl i.e.}\@\xspace}
\nc{\into}{\mathop{\rightarrowtail}}
\nc{\inv}{^{-1}}
\nc{\kk}{k}%{\Bbbk}
\nc{\kkG}{\kk G}% common misprint
\nc{\Loc}[1]{\Locname(#1)}
\nc{\loccit}{{\sl loc.\ cit.}\xspace}
\nc{\Mid}{\,\big|\,}
\nc{\onto}{\mathop{\twoheadrightarrow}}
\nc{\op}{^{\opname}}
\nc{\sminus}{\smallsetminus}
\nc{\potimes}[1]{^{\otimes #1}}% tensor power
\nc{\sbull}{{\scriptscriptstyle\bullet}}
\nc{\SET}[2]{\big\{\,#1\Mid#2\,\big\}}
\nc{\unit}{\mathbb{1}}% unit for \otimes

\dmo{\DPerm}{DPerm}

\nc{\W}{\mathbb{W}}

\dmo{\sta}{sta}% for the Rickard functor and the quotient-to-stab functors
\nc{\rsd}[1]{\mathrm{rsd}_{#1}}% for residue field functors

\dmo{\End}{End}
\nc{\isoto}{\overset{\sim}{\,\to\,}}
\nc{\isofrom}{\overset{\sim}{\,\leftarrow\,}}
\let\le=\leqslant% to have "nicer" \le
\let\ge=\geqslant

\nc{\sto}{\rightsquigarrow}
\nc{\xisoto}[1]{\xrightarrow[\sim]{#1}}
\nc{\xto}[1]{\xrightarrow{#1}}
\nc{\xfrom}[1]{\xleftarrow{#1}}
\nc{\xinto}[1]{\overset{#1}{\,\into\,}}
\nc{\xonto}[1]{\overset{#1}{\,\onto\,}}
\nc{\lto}{\leftarrow}
\nc{\normaleq}{\trianglelefteqslant}
\let\normale\normaleq
\nc{\normal}{\lhd}
\nc{\normaleop}{\mathop{\mathring{\trianglelefteqslant}}}%{\trianglelefteqslant^{\textup{o}}}
\dmo{\chara}{char}%
\dmo{\CoInd}{CoInd}
\dmo{\DMbig}{DM}% Voevodsky's triangulated category of motives
\dmo{\DM}{\DMbig^{\geom}}% Voevodsky's triangulated category of geometric motives
\dmo{\id}{id}
\dmo{\Img}{Im}
\dmo{\im}{im}
\dmo{\Komp}{K}% ground notation for htpy categories
\dmo{\proj}{proj}% f.g. projective modules
\dmo{\rmH}{H}
\dmo{\Res}{Res}
\dmo{\smallb}{b}% ground exponent for ``bounded''
\dmo{\geom}{gm}% ground exponent for ``compact''
\dmo{\stabname}{stab}% stable category of fin. gen. mod.
\dmo{\comp}{comp}
\dmo{\Supp}{Supp}
\dmo{\kosname}{kos}
\dmo{\subname}{Sub}

\nc{\Sub}[1]{\subname_{#1}}
\nc{\Weyl}[2]{{#1}/\!\!/{#2}}%{\overline{#1}_{#2}}%{W_{\!#1}{#2}}% sometimes Weyl is N_G(H)/C_G(H) not mod H.
\nc{\WGH}{\Weyl{G}{H}}
\nc{\tInd}{{}^{\otimes\!}\Ind}% tensor-induction
\nc{\inn}{;}% could be replaced by {\le} or {,} or {;}
\nc{\Vee}[1]{V_{#1}}%
\nc{\Loctens}[1]{\Locname_{\otimes}(#1)}
\nc{\SpcKG}{\Spc(\cK(G))}% most used
\nc{\SpcKGk}{\Spc(\cK(G;\kk))}% most used
\nc{\SpcKE}{\Spc(\cK(E))}% most used
\nc{\Rall}{\rmH^{\sbull\sbull}}%{\rmH_{\scriptscriptstyle\blacktriangle}^{\sbull}}
\nc{\EA}[2]{\mathcal{E}_{#1}(#2)}
\nc{\EApp}[1]{\EA{p}{#1}}

\nc{\kos}[2][]{\kosname_{#1}(#2)}% syntax \kos[ambient group]{subgroup}
\nc{\sKG}{\kos[G]{K}}% most used

\nc{\Zp}{\hat{\bbZ}_p}% p-adic integers

\nc{\adh}[1]{\overline{#1}}% adherence
\nc{\adj}{\dashv}
\nc{\apriori}{{\sl a priori}\xspace}
\nc{\bs}{\backslash}
\nc{\cf}{{\sl cf.}\ }
\nc{\Db}{\Der_{\smallb}}% derived bounded category%\scriptscriptstyle
\nc{\D}{\Der}% derived category
\nc{\FFsep}{\overline{\FF}}
\nc{\Fp}{\bbF_{\!p}}% finite field with p elements
\nc{\gm}{\mathfrak{m}}% prime m
\mathchardef\mhyphen="2D
\nc{\ideal}[1]{\langle #1\rangle}
\nc{\Kb}{\Komp_{\smallb}}% htpy bounded category%\scriptscriptstyle
\nc{\K}{\Komp}% htpy category
\nc{\leop}{\mathop{\mathring{\le}}}%{\le^{\textup{o}}}
\nc{\Lotimes}{\otimes^{\rmL}}
\nc{\To}{\Rightarrow}
\nc{\cSpec}{\mathsf{Spec}}
\nc{\Top}{\mathsf{Top}}

\usepackage{mathtools}
\usepackage{xspace}
\makeatletter
\let\ea\expandafter
\newcount\foreachcount

\def\foreachLetter#1#2#3{\foreachcount=#1
  \ea\loop\ea\ea\ea#3\@Alph\foreachcount
  \advance\foreachcount by 1
  \ifnum\foreachcount<#2\repeat}

\def\definebb#1{\ea\gdef\csname #1#1\endcsname{\ensuremath{\mathbb{#1}}\xspace}}
\foreachLetter{1}{27}{\definebb}

\def\definebb#1{\ea\gdef\csname bb#1\endcsname{\ensuremath{\mathbb{#1}}\xspace}}
\foreachLetter{1}{27}{\definebb}

\makeatother

\date{2023 December 29}
\author{Paul Balmer}
\address{Paul Balmer, UCLA Mathematics Department, Los Angeles, CA 90095, USA}
\email{balmer@math.ucla.edu}
\urladdr{https://www.math.ucla.edu/~balmer}

\author{Martin Gallauer}
\address{Martin Gallauer, Warwick Mathematics Institute, Coventry CV4 7AL, UK}
\email{martin.gallauer@warwick.ac.uk}
\urladdr{https://warwick.ac.uk/mgallauer}

\hypersetup{pdfauthor={Paul Balmer and Martin Gallauer},pdftitle={The spectrum of Artin motives}}

\begin{document}

%------------------------------------------------------------------------------

\title{The spectrum of Artin motives}

\begin{abstract}
We analyze the tt-geometry of derived Artin motives, via modular representation theory of profinite groups.
To illustrate our methods, we discuss Artin motives over a finite field, in which case we also prove stratification.
\end{abstract}


\thanks{First-named author supported by NSF grant~DMS-2153758.
The authors thank the Hausdorff Institute for Mathematics in Bonn for its hospitality during the preparation of this paper.}

\maketitle

\section{Introduction}
\label{sec:intro}

%------------------------------------------------------------------------------

Throughout the paper, $\bbF$ is a `base field' (for motives) and $G$ is a profinite group, \eg the absolute Galois group of~$\bbF$. We denote by $\kk$ a field of `coefficients'.

\smallbreak
\subsection*{Artin motives and permutation modules}
In the big picture of \emph{motivic tt-geometry} one would like to understand the motivic derived category~$\DMbig(\bbF)$ of Voevodsky~\cite{Voevodsky00}, and the motivic stable homotopy category~$\SH(\bbF)$ of Morel-Voevodsky~\cite{morel-voevodsky:SHmot}, from the perspective of tensor-triangular geometry~\cite{balmer:icm}. This project is arguably one of the most challenging parts of tt-geometry and very little is known at the moment~\cite{peter:spectrum-damt,heller-ormsby,gallauer:tannaka-dmet,balmer-gallauer:rage,vishik:isotropic,du-vishik:spc-SH-mot}.
Here we consider another flavor of motives, namely the $\kk$-linear derived category of geometric \emph{Artin motives}
\[
\DAM(\FF;\kk)
\]
\ie the tensor-triangulated subcategory of~$\DM(\bbF;\kk)$ generated by the motives of \emph{finite separable extensions} of~$\bbF$.
As with most tt-categories,~$\DAM(\FF;\kk)$ tends to be `wild' (\Cref{Rem:wild}) and the spectrum~$\Spc(\DAM(\bbF;\kk))$ yields the best classification one can reasonably hope for, namely that of its thick tensor-ideals.
The same is true for~$\DM(\bbF;\kk)$ and we have a surjection on spectra
\[
\Spc(\DM(\bbF;\kk))\onto \Spc(\DAM(\bbF;\kk))
\]
by~\cite{balmer:surjectivity}. Hence the geometric complexity that we describe here for Artin motives also reflects the complexity of~$\DM(\bbF;\kk)$.

We approach Artin motives via representations of the absolute Galois group $\Gal(\bbF^\mathrm{sep}/\bbF)$ of~$\bbF$.
In fact, our results will hold for general profinite groups~$G$.
Our main object of study is the bounded homotopy category of finitely generated permutation $\kk G$-modules (idempotent-completed) over the profinite group~$G$:
\begin{equation}
\label{eq:K(Gamma)}%
\cK(G;\kk)=\Kb(\perm(G;\kk))^{\natural}.
\end{equation}
Voevodsky established in~\cite{Voevodsky00} an equivalence between this category $\cK(G;\kk)$ for~$G=\Gal(\bbF^{\mathrm{sep}}/\bbF)$ and our initial category~$\DAM(\FF;\kk)$.
See \Cref{Rec:perm,Rec:Artin-motives}.
So although our motivation comes from algebraic geometry, we write most of the paper in the more general setting of representation theory.

There are further incarnations of the tt-category~\eqref{eq:K(Gamma)}, in terms of cohomological Mackey functors, or equivariant modules over the Eilenberg-MacLane spectrum of the constant Mackey functor. See for instance~\cite{balmer-gallauer:Dperm}.

\smallbreak
\subsection*{The spectrum}
We already computed in~\cite{balmer-gallauer:TTG-Perm} the spectrum of the permutation category $\cK(G;\kk)$ when $G$ is a \emph{finite} group.
So our first task is to `pass these results to the limit'.
This relies on the general continuity result of~\cite{gallauer:tt-fmod}, yielding for every profinite group~$G$ a homeomorphism (\Cref{Prop:Spc-K(G)-colim})
\begin{equation}
\label{eq:Spc-lim-intro}%
\Spc(\cK(G;\kk))\cong\lim_{N\normaleop G}\ \Spc(\cK(G/N;\kk))
\end{equation}
where $N\normaleop G$ stands for the open normal subgroups of~$G$ and where the transition maps between the spectra~$\Spc(\cK(G/N;\kk))$ on the right-hand side of~\eqref{eq:Spc-lim-intro} are induced by inflation between the finite groups~$G/N$.
With this, we can describe all points of the spectrum of $\cK(G;\kk)$ for all profinite groups~$G$ as follows:
\begin{Thm}
\label{Thm:all-pts-intro}%
If $\chara(\kk)=0$ then $\SpcKGk$ is a single point, \ie the only tt-prime is zero.
If $\chara(\kk)=p>0$ then every tt-prime in~$\cK(G;\kk)$ is of the form~$\cP_G(H;\gp)$ for a closed pro-$p$-subgroup~$H\le G$ and a homogeneous prime~$\gp$ in the cohomology of the Weyl group of~$H$, for a unique pair~$(H,\gp)$ up to $G$-conjugation.
\end{Thm}

These primes~$\cP_G(H;\gp)$ are explicit.
Let us denote by $\WGH=(N_G(H))/H$ the Weyl group of the subgroup~$H$ of~$G$.
As in the finite group case, there exists a tt-functor called the \emph{modular $H$-fixed points functor}~$\Psi^H\colon\cK(G;\kk)\to \cK(\WGH;\kk)$ that $\kk$-linearizes the $H$-fixed points on finite $G$-sets (\Cref{Cons:Psi^H-pro}).
We can compose it with the obvious functor to the derived category
\[
\cK(G;\kk)\to \cK(\WGH;\kk)\onto \Db(\kk(\WGH))
\]
and get a tt-functor from our $\cK(G;\kk)$ to the derived category of the Weyl group of~$H$.
Via this tt-functor, we can pull-back every tt-prime of $\Db(\kk(\WGH))$; the latter correspond to homogeneous primes~$\gp$ of the cohomology of~$\WGH$.
These pull-backs are our primes~$\cP_G(H;\gp)$.
The statement of \Cref{Thm:all-pts-intro} is that every prime of~$\cK(G;\kk)$ is of this form, for a unique choice of~$H$ and~$\gp$ up to $G$-conjugation.

We outline in \Cref{Rem:outline}, without elaborating, some alternative descriptions of the topology of~$\SpcKGk$ via pro-elementary abelian groups or via twisted cohomology. These themes have been explored in depth in~\cite{balmer-gallauer:TTG-Perm} in the finite case and would appear repetitive here.
We find it more instructive to discuss an example in some detail.

\smallbreak
\subsection*{Artin motives over a finite field}
Let us illustrate our methods by computing the spectrum of geometric Artin motives over a \emph{finite} base field~$\bbF$.
In other words, we consider $\cK(G;\kk)$ for $G=\hat{\bbZ}$, the profinite integers.
See \Cref{sec:artin-finite-fields}.

Assume $\kk$ is a field of characteristic $p>0$.
The $p$-Sylow of any finite quotient $G/N$ is cyclic of order $p^n$, for some $n\geq 0$.
We proved in~\cite{balmer-gallauer:TTG-Perm} that the spectrum of $\cK(G/N;\kk)$ is then the space
\[
\W^n=\qquad
\vcenter{\xymatrix@R=1em@C=.7em{
{\scriptstyle\gm_0\kern-1em}
& {\bullet} \ar@{-}@[Gray] '[rd] '[rr] '[drrr]
&&
{\bullet}
&{\kern-1em\scriptstyle\gm_1}
&&{\scriptstyle\gm_{n-1}\kern-1em}&\bullet&&\bullet&{\kern-1em\scriptstyle\gm_n}
\\
&{\scriptstyle\gp_1\kern-1em}& {\bullet}&&& {\cdots}& \ar@{-}@[Gray] '[ru] '[rr] '[rrru] &&\bullet&{\kern-1em}\scriptstyle\gp_{n}}}
\]
consisting of $2n+1$ points, with specialization relations going upward.
(We also denote by $\W^\infty=\colim_n\W^n$ the union of these along the canonical inclusions.)
Taking the limit of the~$\W^n$ as in~\eqref{eq:Spc-lim-intro} yields:

\begin{Thm}
\label{Thm:Fq-intro}%
Let $\bbF$ be a finite field.
If $\chara(\kk)=0$ then $\Spc(\DAM(\FF;\kk))=\ast$. If $\chara(\kk)>0$ then $\Spc(\DAM(\FF;\kk))$ is the Alexandroff extension of~$\W^\infty$:
\begin{equation}
\label{eq:Spc(DAM-mod-p)-intro}%
\vcenter{\xymatrix@R=1em@C=.7em{
{\scriptstyle\gm_0\kern-1em}
& {\bullet} \ar@{-}@[Gray] '[rd] '[rr] '[drrr]
&&
{\bullet}
&{\kern-1em\scriptstyle\gm_1}
&{\cdots}&{\scriptstyle\gm_{n-1}\kern-1em}&\bullet&&\bullet&{\kern-1em\scriptstyle\gm_n}&{\cdots}&&\bullet&{\kern-1em\scriptstyle\gm_{\infty}}
\\
&{\scriptstyle\gp_1\kern-1em}& {\bullet}&&& {\cdots}& \ar@{-}@[Gray] '[ru] '[rr] '[rrru] '[rrrr] &{\scriptstyle\gp_{n}}\kern-1em&\bullet &&&{\cdots}}}
\end{equation}
A subset $V$ is Thomason if and only if (i) $V$ is stable under specializations and (ii) $\gm_\infty\in V$ implies $V$ cofinite.
These subsets correspond bijectively with thick tensor-ideals of~$\DAM(\FF;\kk)$.
In particular, there are continuum many.
\end{Thm}

The integral spectrum $\Spc(\DAM(\FF;\bbZ))$ follows by analyzing the fibers of the continuous map to $\Spec(\bbZ)$. Above the generic point~$(0)$, the fiber is a single point. Above each closed point of~$\Spec(\bbZ)$ the fiber is a copy of~\eqref{eq:Spc(DAM-mod-p)-intro}.
See~\Cref{Cor:Fq-stratification-over-Z}.

\smallbreak
\subsection*{Stratification}
We can consider the big derived category of permutation modules~$\DPerm(G;\kk)$, that admits our~$\cK(G;\kk)$ as its compact objects.
In the case of a finite group~$G$, we proved in \cite[\S\,9]{balmer-gallauer:TTG-Perm} that~$\DPerm(G;\kk)$ satisfies \emph{BHS-stratification}, following Barthel-Heard-Sanders~\cite{barthel-heard-sanders:stratification-Mackey}. This means that its spectrum is reasonable, namely is a weakly noetherian space (it is even noetherian for $G$ finite), and that the ensuing support theory on the big objects of~$\DPerm(G;\kk)$ classifies all localizing tensor-ideals. (See \Cref{Rec:BHS-stratification}.)
It is natural to ask whether this result also `passes to the limit' to profinite groups. This theoretical problem of the behavior of stratification under colimit has not been solved in glorious generality. We expect it to be non-trivial since big tt-categories for which stratification fails can still be the colimit of tt-categories that satisfy BHS-stratification (\eg\ for non-noetherian commutative rings viewed as the colimit of their noetherian subrings).
In fact, for a general profinite group~$G$, the spectrum~$\Spc(\cK(G;\kk))$ may even cease to be weakly noetherian. See~\Cref{Exa:non-weakly-noetherian}.

Nevertheless, for $G=\hat{\bbZ}$ and $\chara(\kk)=p$, we obtain the spectrum of \Cref{Thm:Fq-intro} and although it is not noetherian (\Cref{Rem:Zp-not-noeth}), it happens to still be \emph{weakly} noetherian (\Cref{Rem:Spc-Z_p-weakly-noetherian}).
And we do prove in \Cref{Thm:Z_p-stratification} that $\DPerm(G;\kk)$ is indeed BHS-stratified in that special case.
This is a little miracle of procyclic groups. Putting things together, we obtain:
\begin{Thm}
\label{Thm:Fq-stratification-intro}%
Let $\bbF$ be a finite field and $\kk$ an arbitrary field.
Then the big derived category of Artin motives~$\DAMbig(\bbF;\kk)$ is stratified in the sense of~\cite{barthel-heard-sanders:stratification-Mackey}{\rm:} The spectrum of its compacts~$\DAM(\bbF;\kk)$ is weakly noetherian, and its localizing tensor-ideals are in one-to-one correspondence with the subsets of~$\Spc(\DAM(\bbF;\kk))$.

Moreover, it satisfies the telescope property: The smashing localizing ideals are in one-to-one correspondence with the finite ones.
\end{Thm}

The outline of the paper is straightforward.
We recall basics in \Cref{sec:basics} and discuss modular fixed-points over profinite groups.
We prove the homeomorphism~\eqref{eq:Spc-lim-intro} and deduce \Cref{Thm:all-pts-intro} in \Cref{sec:spectrum-KG}.
That section also contains general topological properties of the maps induced by restriction and modular fixed-points.
In \Cref{sec:procyclic}, we illustrate our work by computing the spectrum in the example of procyclic groups.
We then turn attention to BHS-stratification and gather a couple of preparatory tt-results in \Cref{sec:stratification}, that could be of independent interest.
In \Cref{sec:stratification-procyclic}, we prove BHS-stratification in the procyclic case.
We conclude the article with the translation to Artin motives over finite fields in the brief \Cref{sec:artin-finite-fields}.

\begin{Ack}
We thank an anonymous referee for a thorough reading of the text.

For the purpose of open access, the authors have applied a Creative Commons Attribution (CC-BY) licence to any Author Accepted Manuscript version arising from this submission.
\end{Ack}
% ------------------------------------------------------------------------------

\section{Basics}
\label{sec:basics}%

%------------------------------------------------------------------------------

%
\begin{Hyp}
\label{Hyp:general}%
Unless otherwise stated, $G$ is a profinite group and $\kk$ is a field of positive characteristic~$p$.
\end{Hyp}

\begin{Conv}
All subgroups $H\le G$ will be closed.
We write $H\leop G$ for open subgroups (that is, closed and of finite index).
We denote by $\Sub{}(G)$ the poset of (closed) subgroups.
\end{Conv}

\begin{Rec}
\label{Rec:perm}%
To every open subgroup~$H\leop G$ there is an associated finite-dimensional (left) $\kk G$-module~$\kk(G/H)$.
The underlying $\kk$-vector space has~$G/H$ as basis, with obvious left $G$-action.
It belongs to the category~$\Mod(G;\kk)$ of discrete~$(G;\kk)$-modules, cf.~\cite[Section~2]{balmer-gallauer:Dperm} for details.
We let $\perm(G;\kk)$ be the additive closure of all these objects~$\kk(G/H)$, $H\leop G$, in~$\Mod(G;\kk)$.
It is the category of (finitely generated) \emph{permutation modules} over~$G$ with coefficients in~$\kk$.
We recall from~\eqref{eq:K(Gamma)} that the main character in this paper is the rigid tt-category
\begin{equation*}
\cK(G)=\cK(G;\kk):=\Kb(\perm(G;\kk))^{\natural}=\Kb(\perm(G;\kk)^{\natural})
\end{equation*}
obtained from~$\perm(G;\kk)$ by taking bounded complexes and maps up to homotopy and by idempotent-completing (in any order).
We often abbreviate~$\cK(G)$ for~$\cK(G;\kk)$, to lighten notation.
\end{Rec}

\begin{Prop}
\label{Prop:K(G)-union}
Let $(G_\alpha)_\alpha$ be an inverse system of profinite groups (\eg an inverse system of finite groups) with surjective transition maps, and denote by~$G=\lim_\alpha G_\alpha$ its limit.
The inflation functor $\cK(G_\alpha)\to\cK(G)$ is fully faithful and the category~$\cK(G)$ is the union of the tt-categories~$\cK(G_\alpha)$.
\end{Prop}
\begin{proof}
Inflation functors between finite groups are fully-faithful, as usual.
This implies the same result for profinite groups.
And any object in~$\cK(G)$ comes through inflation from $G/N$, for some $N\normaleop G$.
The projection $G\to G/N$ factors through some $G_\alpha$.
This proves the proposition.
\end{proof}

One can Ind-complete~$\cK(G)$:
\begin{Not}
\label{Not:DPerm}%
Following~\cite[Section~3]{balmer-gallauer:Dperm}, we denote by
\[
\cT(G):=\DPerm(G;\kk)
\]
the localizing subcategory generated by~$\cK(G)$ inside~$\K(\Mod(G;\kk))$.
It is a rigidly-compactly generated `big' tt-category with compact part $\cK(G)$.
\end{Not}

\begin{Def}
\label{Def:index-prime-to-p}%
We say that a closed subgroup $H\le G$ is of \emph{index prime to~$p$} if every open subgroup~$K\leop G$ containing $H$ has (usual, finite) index prime to~$p$.
We say that $G$ has \emph{order prime to~$p$} if $H=1$ has index prime to~$p$, meaning that~$G$ is an inverse limit of finite groups of order prime to~$p$.
\end{Def}

\begin{Rec}
\label{Rec:p-Sylow}%
A minimal subgroup $H\le G$ of index prime to~$p$ is called a \emph{$p$-Sylow subgroup}.
A profinite group~$G$ is a \emph{pro-$p$-group} if it is an inverse limit of finite $p$-groups, or equivalently, if every finite quotient of~$G$ is a $p$-group.
A $p$-Sylow subgroup is necessarily a pro-$p$-group.
It is maximal among subgroups with this property.
In fact, much of the Sylow theory familiar from finite groups generalizes to profinite groups.
We refer to~\cite[Chapter~2]{ribes-zalesskii:profinite} for more details.
\end{Rec}

\begin{Rem}
\label{Rem:Spc-to-Sylow}%
Let $H\le G$ be a subgroup of index prime to~$p$.
If~$G$ is finite, the restriction functor $\Res^G_H\colon \cK(G)\to \cK(H)$ is faithful (by the usual argument) and, as a consequence, the map~$\rho_H=\Spc(\Res^G_H)\colon \Spc(\cK(H))\onto \Spc(\cK(G))$ induced on spectra is surjective, see~\cite[Proposition~3.8]{balmer-gallauer:TTG-Perm}.
One deduces the same statements for~$G$ profinite as everything is inflated from a finite quotient.

Similarly, if $\chara(\kk)=0$, we obtain the first claim of \Cref{Thm:all-pts-intro}.

We will now show that the restriction functor $\Res^G_H\colon \cT(G)\to \cT(H)$ is also faithful on the big categories.
Let us write $\CoInd_H^G$ for its right adjoint.
\end{Rem}
\begin{Prop}
\label{Prop:Res-split-faithful}
Let $H\le G$ be of index prime to~$p$.
Then $\unit_{\cT(G)}$ belongs to $\Loctens{\CoInd_H^G(\unit)}$ and the unit $\eta\colon\Id\to \CoInd^G_H\Res^G_H$ admits a retraction.
In particular, the functor $\Res^G_H:\cT(G)\to\cT(H)$ is (split) faithful.
\end{Prop}
\begin{proof}
Let $\xi\colon t\to \unit$ be a homotopy fiber in~$\cT(G)$ of the unit~$\eta_{\unit}\colon\unit \to \CoInd_H^G(\unit)$. We have $\Res^G_H(\xi)=0$ and since $\Res^G_H$ is faithful on compacts, this map $\xi$ must be a \emph{phantom map}: For every compact $c$, and every map $c\to t$ the composite $c\to t\xto{\xi}\unit$ is zero. On the other hand, $\unit=\kk[0]$ is easily seen to be \emph{endofinite}, \ie for every compact $c\in\cT(G)^c$, the group~$\Hom(c,\unit)$ is a module of finite length over the ring~$\End(\unit)=\kk$, simply because $\Hom(c,\unit)$ is finite dimensional. It is a general fact that phantom maps into endofinite objects are zero; see~\cite[Theorem~1.2\,(3)]{krause:decomposing}. Thus $\xi=0$ and the unit~$\eta_{\unit}$ is a split monomorphism. By the projection formula, the same holds for~$\eta\colon \Id\to \CoInd^G_H\Res^G_H\cong\CoInd^G_H(\unit)\otimes\Id$.
\end{proof}
\begin{Rem}
One can also construct a retraction of $\unit\to \CoInd(\unit)$ by hand, via a suitable homotopy colimit of the retractions at each finite stage.
\end{Rem}

Let us now turn to modular fixed-points.
\begin{Cons}
\label{Cons:Psi^H-pro}%
We may generalize the construction of modular fixed-points from \cite{balmer-gallauer:TTG-Perm}.
Let $N\normale G$ be a normal closed pro-$p$-subgroup.
The analogues of~Lemma~5.3 and~Proposition~5.4 in \cite{balmer-gallauer:TTG-Perm} are true and we may perform Construction~5.6 in \loccit
Indeed, let $H,K\leop G$ be open subgroups such that $N\le H$ and $N\not\le K$. By~\cite[Proposition~2.1.4(c)]{ribes-zalesskii:profinite}, there exists an open normal subgroup $U\normaleop G$ such that $NU\le H$. Choose an open normal subgroup~$M\normaleop G$ contained in $NU\cap K$. Then $NM\le H$ and $NM\not\le K$ and $NM/M\cong N/N\cap M$ is a finite $p$-group. One can therefore apply~Lemma~5.3 in \cite{balmer-gallauer:TTG-Perm} to $G/M$, $H/M$, $K/M$, $NM/M$ and subsequently use inflation to~$G$.
In summary, the constructions of \loccit yield the modular fixed-points functor $\Psi^{N\inn G}\colon \cT(G)\to \cT(G/N)$ for normal closed pro-$p$-subgroups~$N\normale G$.

If $H\le G$ is a closed pro-$p$-subgroup that is not necessarily normal then its normalizer is also closed and hence profinite.
We may therefore repeat \cite[Definition~5.10]{balmer-gallauer:TTG-Perm}.
In other words, we now have coproduct-preserving tt-functors
\[
\Psi^{H\inn G}\colon \cT(G)\to \cT(\Weyl{G}{H}),
\]
still called \emph{modular $H$-fixed-points}.
They are characterized by the facts that on permutation modules we have $\Psi^{H}(\kk(X))\cong \kk(X^H)$ for every finite discrete $G$-set~$X$ and these functors are applied degreewise on complexes.
\end{Cons}

\begin{Rem}
\label{Rem:Psi-colim}
Let $H\le G$ be a closed pro-$p$-subgroup and let $N\normaleop G$.
It is straightforward to check that the following square commutes (already at the level of additive categories of permutation modules):
\[
\xymatrix@C=5em{\cT(G)
  \ar[r]^{\Psi^{H}}
  &
  \cT(\Weyl{G}{H})
  \\
  \cT(G/N)
  \ar[u]_{\Infl}
  \ar[r]_-{\Psi^{(HN)/N}}
  &
  \cT(\Weyl{G}{(HN)})
  \ar[u]_{\Infl}
}
\]
Therefore, in combination with \Cref{Prop:K(G)-union}, we could also define the modular fixed-points functor~$\Psi^H$ for profinite groups as the unique extension of modular fixed-points for finite groups introduced in~\cite{balmer-gallauer:TTG-Perm}.
\end{Rem}

%------------------------------------------------------------------------------

\section{General profinite groups}
\label{sec:spectrum-KG}%

%------------------------------------------------------------------------------

Recall our standing \Cref{Hyp:general}.

\subsection{The spectrum}
We begin our discussion of the tt-geometry of $\cK(G)=\Kb(\perm(G;\kk))^\natural$ by applying the general continuity result of~\cite[Section~8]{gallauer:tt-fmod}.

\begin{Prop}
\label{Prop:Spc-K(G)-colim}%
Let $(G_\alpha)_\alpha$ be an inverse system of profinite groups with surjective transition maps, and denote by~$G=\lim_\alpha G_\alpha$ its limit.
On spectra, this yields a homeomorphism with the inverse limit
\begin{equation}
\label{eq:Spc-K(G)-colim}%
\Spc(\cK(G))=\lim_{\alpha}\ \Spc(\cK(G_\alpha))
\end{equation}
where the transition maps are surjective.
\end{Prop}
\begin{proof}
By \Cref{Prop:K(G)-union}, we have $\cK(G)=\colim_\alpha\cK(G_\alpha)$ with fully faithful transition functors.
By~\cite[Proposition~8.2]{gallauer:tt-fmod}, the functor $\Spc(-)$ turns such directed colimits of tt-categories into limits of spaces.
Surjectivity uses that inflation is faithful and~\cite[Corollary~1.8]{balmer:surjectivity}.
\end{proof}

\begin{Rec}
Keep the notation, $G=\lim_\alpha G_\alpha$ an inverse limit of profinite groups.
For the set~$\Sub{}(G)$ of closed subgroups we have the following bijection:
\begin{equation}
\label{eq:Sub-continuous}
\Sub{}(G)\isoto\lim_{\alpha}\Sub{}(G_\alpha).
\end{equation}
\end{Rec}

Let us write $\Sub{}(G)/G$ for the quotient of~$\Sub{}(G)$ under $G$-conjugation.
For the reader's convenience we recall the proof of the following result.
\begin{Lem}[{\cite[pp.~B8--9]{dress:lectures-notes-repfin-I}}]
\label{Lem:Sub-mod-G-continuous}
The canonical map is a bijection:
\begin{equation*}
\label{eq:Sub-mod-G-continuous}
\Sub{}(G)/G\isoto\lim_{\alpha}\left(\Sub{}(G_\alpha)/G_\alpha\right).
\end{equation*}
\end{Lem}
\begin{proof}
It is enough to show this for a fundamental system of open normal neighborhoods of~$1$, in other words, for $G=\lim_NG/N$.
Let us show injectivity first.
Let $C,D\in \Sub{}(G)$ such that $CN/N\sim_{G/N}DN/N$ for each~$N$.
Consider for every $N$ the subset~$X_N:=\SET{g}{C^gN=DN}$ of~$G/N$.
These define a sub-inverse system of~$\{G/N\}_N$.
Since each $X_N$ is finite and non-empty, the limit is non-empty.
Any $g\in \lim_NX_N\subseteq \lim_N(G/N)=G$ has the property that $C^g=D$.
It remains to show surjectivity.
So let $(H_N)_N\in\lim_N\left(\Sub{}(G/N)/G\right)$.
Consider for every $N$ the subset~$Y_N:=\SET{(H_N)^g}{g\in G}$ of~$\Sub{}(G/N)$.
These define a sub-inverse system of~$\{\Sub{}(G/N)\}_{N}$.
Since each $Y_N$ is finite and non-empty, the limit is non-empty.
Any $C\in\lim_NY_N\subseteq\lim_N\Sub{}(G/N)=\Sub{}(G)$, by~\eqref{eq:Sub-continuous}, has the property that $CN\sim_{G/N}H_N$ for each~$N$.
\end{proof}

\begin{Rem}
\label{Rem:Vee-Dirac}
For every profinite group~$G$, we write
\[
\Vee{G}:=\Spc(\Db(\kk G))
\]
for the cohomological support space, with the common abuse of notation of writing~$\Db(\kkG)$ for the bounded derived category of finite $\kk$-dimensional discrete~$(G;\kk)$-modules.
It is an easy consequence of the finite case that the cohomological support space underlies a Dirac scheme, namely the canonical comparison map at the top of the commutative square
\[
\xymatrix@R=2em@C=3em{
 \Vee{G}
 \ar[r]^-{\comp}_-{\simeq}
 \ar[d]_{\simeq}
 &
 \Spech(\Hm(G;\kk))
 \ar[d]_{\simeq}
 \\
 \lim_N\Vee{G/N}
 \ar[r]^-{\comp}_-{\simeq}
 &
 \lim_N\Spech(\Hm(G/N;\kk))
 }
\]
is a homeomorphism.
See~\cite[Proposition~6.5]{gallauer:tt-dtm-algclosed}.
\end{Rem}

\begin{Not}
Let $H\le G$ be a closed subgroup. Its Weyl group~$\Weyl{G}{H}$ is again a profinite group.
Assume further that $H$ is a pro-$p$-group.
The composite
\[
\check{\Psi}^H\colon \cK(G)\xto{\Psi^H}\cK(\Weyl{G}{H})\xto{\Upsilon}\Db(\kk(\Weyl{G}{H}))
\]
of modular $H$-fixed-points and the canonical functor (that is the identity on objects)
induces on spectra a continuous map
\begin{equation}
\label{eq:check-psi}
\check{\psi}^H:=\Spc(\check{\Psi}^H)\colon\Vee{\Weyl{G}{H}}\to\SpcKG.
\end{equation}
\end{Not}

\begin{Not}
Let $H\le G$ be a closed pro-$p$-subgroup and $\gp\in\Vee{\Weyl{G}{H}}$.
We write
\[
\cP_G(H,\gp):=\check{\psi}^H(\gp)
\]
for the corresponding point of~$\SpcKG$.
We just write $\cP(H;\gp)$ when $G$ is clear.
\end{Not}

\begin{Cor}
\label{Cor:SpcKG-decomposition}
The maps~\eqref{eq:check-psi} are injective and induce a decomposition as sets
\[
\SpcKG=\coprod_{(H)}\Vee{\Weyl{G}{H}}
\]
over conjugacy classes~$(H)$ of closed pro-$p$-subgroups~$H$ of~$G$.
\end{Cor}
\begin{proof}
Let $x\in\SpcKG$.
By \Cref{Prop:Spc-K(G)-colim} and the finite case \cite[Proposition~7.32]{balmer-gallauer:TTG-Perm}, we may write $x=(\cP_{G/N}(H_N/N,\gp_N))_N$ as a compatible family of points in $\Spc(\cK(G/N))$, indexed by $N\normaleop G$, where \mbox{$N\le H_N\le G$} and $H_N/N$ is a $p$-group, and $\gp_N\in\Vee{\Weyl{(G/N)}{(H_N/N)}}$.
For all $N\le N'$ inflation sends $\cP_{G/N}(H_N/N,\gp_N)$ to~$\cP_{G/N'}(N'H_N/N',\bar{\pi}^{G/N'}\gp_N)$, see~\cite[Remark~7.6\,(c)]{balmer-gallauer:TTG-Perm}.
Hence $N'H_N\sim_G H_{N'}$, and we conclude from \Cref{Lem:Sub-mod-G-continuous} that there exists a closed subgroup~$H\le G$ with $NH\sim_G H_N$ for all~$N\normaleop G$.
Moreover, $H$ is a pro-$p$-group, and is unique up to conjugation.
It follows from \Cref{Rem:Vee-Dirac} that the~$\gp_N$ come from a unique point~$\gp$ in~$\Vee{\Weyl{G}{H}}$.
(One uses that the normalizer of~$HN/N$ in~$G/N$ is given by $N_G(H)N/N$ and the commutative square in \Cref{Rem:Psi-colim}.)
\end{proof}

\begin{Rem}
In other words, we have proved that all points in $\SpcKG$ are of the form~$\cP_G(H,\gp)$.
Moreover, $\cP_G(H,\gp)=\cP_G(H',\gp')$ if and only if there exists $g\in G$ such that $H'=H^g$ and $\gp'=\gp^g$, cf.~\cite[Theorem~7.16]{balmer-gallauer:TTG-Perm}.
\end{Rem}

\subsection{The map induced by restriction}
Let us now discuss the topological properties of the map~$\rho_H=\rho^G_H:=\Spc(\Res^G_H)$.
\begin{Prop}
\label{Prop:rho-closed}
Let $H\le G$ be a closed subgroup.
Then $\rho_H\colon\Spc(\cK(H))\to\SpcKG$ is a closed map with image
\[
\bigcap_{H\le_G K\leop G}\supp(\kk(G/K)).
\]
\end{Prop}

Before proving the proposition, we recall a point-set topology fact.
\begin{Lem}
\label{Lem:closed-limit}
Let $(X_i), (Y_i)$ be inverse systems in the category~$\cSpec$ of spectral spaces and spectral maps. Let $X$ and~$Y$ be their respective limits\,\textup{(}\footnote{Recall that limits in~$\cSpec$ are computed in~$\Top$.}\textup{)}, and let $f_i\colon X_i\to Y_i$ be a natural transformation with limit $f\colon X\to Y$. Assume each $f_i$ is closed and has finite fibers.
Denote by~$\pi_i\colon Y\to Y_i$ the canonical projection. Then the map~$f$ is closed with image $\cap_i\pi_i\inv(\Img(f_i))$.
\end{Lem}
\begin{proof}
Let $Z=\cap_i\pi_i\inv(\Img(f_i))\subseteq Y$.
Since $\pi_i(f(x))=f_i(\pi_i(x))$ it is clear that $\Img(f)\subseteq Z$.
For the converse, let $y\in Z$.
Consider, for each $i$, the subset $C_i:=\SET{x\in X_i}{f_i(x)=\pi_i(y)}$ of~$X_i$.
By assumption, the fiber $C_i=f_i\inv(\{\pi_i(y)\})$ is finite and it is non-empty because $y\in Z$.
It is also a sub-inverse system of~$X_\sbull$ hence its limit is non-empty.
Any $x\in\lim_i C_i\subseteq X$ maps to~$y$ under~$f$.

To show that~$f$ is a closed map it suffices to prove that it has the going-up property~\cite[Theorem~5.3.3]{dst:spectral-spaces}.
That is, let $x\in X$ and $f(x)\sto y\in Y$ a specialization.
We need to exhibit $x'\in X$ such that $x\sto x'$ and $f(x')=y$.
As the projection $\pi_i\colon Y\to Y_i$ is continuous, we also have $f_i(\pi_i(x))=\pi_i f(x)\sto \pi_i(y)$.
Consider, for each $i$, the subset $D_i:=\SET{x'_i\in X_i}{\pi_i(x)\sto x'_i, f_i(x'_i)=\pi_i(y)}$ of~$X_i$.
By assumption, $D_i$ is non-empty and finite since $D_i\subseteq f_i\inv(\{\pi_i(y)\})$.
The $(D_i)_i$ form a sub-inverse system of~$X_\sbull$ and their limit $D=\lim D_i$ is non-empty.
Any $x'\in D\subseteq X$ has the property that $f(x')=y$ and $x\sto x'$; the former is obvious and the latter follows from~\cite[Theorem~2.2.1 or~2.3.9]{dst:spectral-spaces}.
\end{proof}

\begin{proof}[Proof of \Cref{Prop:rho-closed}]
Let $I=\{N\normaleop G\}$ be the directed poset of open normal subgroups, so that $\Spc(\cK(G))=\lim_{N\normaleop G}\Spc(\cK(G/N))$ by \Cref{Prop:Spc-K(G)-colim}.
Since $H=\lim_{N\normaleop G}H/(N\cap H)$, we have $\Spc(\cK(H))=\lim_{N\normaleop G}\Spc(\cK(H/N\cap H))$ as well and it suffices to verify that the maps~$\rho^{G/N}_{H/H\cap N}\colon \Spc(\cK(H/N\cap H))\to \Spc(\cK(G/N))$ satisfy the hypotheses of the maps~$f_i$ in \Cref{Lem:closed-limit}, for every open normal subgroup~$N\normaleop G$. As $H/(N\cap H)$ is a subgroup of~$G/N$, we are in finite-group territory and can invoke~\cite[Proposition~4.7 and Corollary~4.9]{balmer-gallauer:TTG-Perm}.
\end{proof}

\begin{Rem}
One can in fact `pass the coequalizers of~\cite[Proposition~4.7]{balmer-gallauer:TTG-Perm} to the limit', in the hopefully obvious sense. As this will not be used here, we leave this development to the interested reader.
\end{Rem}

\begin{Exa}
\label{Exa:pronilpotent}%
Suppose that $G$ is a profinite group such that the inclusion of the pro-$p$-Sylow $P\into G$ has a retraction.
In other words, $G=H\rtimes P$ (for $H$ of order prime to~$p$). This happens in particular for pro-nilpotent groups, hence for abelian profinite groups.
See~\cite[Proposition~2.3.8]{ribes-zalesskii:profinite}.
In that case, the map on spectra induced by restriction to~$P$ admits a continuous retraction.
Since it is also surjective by \Cref{Rem:Spc-to-Sylow}, it is a homeomorphism
\[
\rho_{P}=\Spc(\Res^G_{P}):\Spc(\cK(P))\isoto\Spc(\cK(G)).
\]
\end{Exa}

\subsection{The map induced by modular fixed-points}
Let us now discuss the topological properties of the map~$\psi^H=\psi^{H\inn G}:=\Spc(\Psi^{H\inn G})$.

For a finite group~$G$ and any subgroup $K\le G$ we defined a Koszul object $\sKG=\tInd_K^G(0\to \kk\xto{1}\kk\to 0)$ using tensor-induction in~\cite[Construction~3.14]{balmer-gallauer:TTG-Perm}.
Note that this construction also makes perfect sense for $G$ pro-finite as long as $K\leop G$ is an open subgroup, that is, of finite index.
We then generalize~\cite[Lemma~7.12]{balmer-gallauer:TTG-Perm}:
\begin{Lem}
\label{Lem:P-to-s-pro}%
Let $H\le G$ be a pro-$p$-subgroup and $\gp\in \Vee{\WGH}$.
Let $K\leop G$ be an open subgroup and~$\sKG=\tInd_K^G(0\to \kk\xto{1}\kk\to 0)$.
Then $\sKG\in\cP_G(H,\gp)$ if and only if~$H\le_G K$.
\end{Lem}
\begin{proof}
Let $N\normaleop K$ be an open normal subgroup of~$G$.
By \cite[Lemma~3.17]{balmer-gallauer:TTG-Perm}, we have $\sKG=\Infl(\kos[G/N]{K/N})$. Consequently, $\sKG\in\cP_G(H,\gp)$ if and only if $\kos[G/N]{K/N}\in\cP_{G/N}(HN/N,\gq)$ where $\gq\in\Vee{\Weyl{G}{HN}}$ is the image of~$\gp$ under inflation.
We now apply \cite[Lemma~7.12]{balmer-gallauer:TTG-Perm}  which tells us that this is equivalent to $HN/N\le_G K/N$, or $H\le_G K$, as claimed.
\end{proof}

\begin{Cor}
If $\cP_G(H,\gp)\subseteq \cP_G(H',\gp')$ then $H'\le_G H$.
\qed
\end{Cor}

\begin{Lem}
\label{Lem:psi-N-closed}
Let $H\normale G$ be a normal pro-$p$-subgroup.
Then the map $\psi^H$ is a closed immersion $\Spc(\cK(G/H))\hook\SpcKG$ with image
\begin{equation}
\label{eq:psi-H-image}
\bigcap_{H\not\le_G K\leop G}\supp(\kos[G]{K}).
\end{equation}
\end{Lem}
\begin{proof}
This map admits a retraction induced by inflation (cf.~\cite[Corollary~5.16]{balmer-gallauer:TTG-Perm}).
Therefore it will be a closed immersion once we know the image is closed.
So it suffices to show the second statement.
For each $N\normaleop G$, the corresponding map~$\psi^{HN/N}\colon \Spc(\cK(HN/N))\to\Spc(\cK(G/N))$ is a closed immersion, by~\cite[Proposition~7.18]{balmer-gallauer:TTG-Perm}, with image
\begin{equation}
\label{eq:psi-HN-image}
\bigcap_{HN/N\not\le_G K/N}\supp(\kos[G/N]{K/N}).
\end{equation}
By \Cref{Rem:Psi-colim}, the map~$\psi^H$ in the statement is the limit of these~$\psi^{HN/N}$ along inflation and we may therefore apply \Cref{Lem:closed-limit}.
It tells us that the image of~$\psi^H$ is the intersection of all subsets~\eqref{eq:psi-HN-image} inflated to~$\cK(G)$, that is,
\[
\bigcap_{N\normaleop G}\ \bigcap_{HN/N\not\le_G K/N}\supp(\kos[G]{K}).
\]
For every $H\not\le_G K\leop G$ there exists an open normal subgroup $N\normaleop G$ such that $N\le K$ and $HN/N\not\le_G K/N$, by \Cref{Lem:Sub-mod-G-continuous}.
It follows that the subset described by this double intersection is precisely~\eqref{eq:psi-H-image}.
\end{proof}

\begin{Cor}
\label{Cor:psi-H-closed}
Let $H\le G$ be a pro-$p$-subgroup, not necessarily normal.
Then the map~$\psi^H\colon\Spc(\cK(\WGH))\to\SpcKG$ is a closed map.
\end{Cor}
\begin{proof}
Since $\psi^{H\inn G}=\rho^G_{N_G(H)}\circ \psi^{H\inn N_G(H)}$ the result follows from \Cref{Prop:rho-closed} for restriction to the normalizer and \Cref{Lem:psi-N-closed} for the case of $H\normale N_G(H)$.
\end{proof}

\begin{Rem}
\label{Rem:outline}%
In the finite group case, we combined the two techniques discussed above, namely restriction and modular fixed-points, to prove that $\Spc(\cK(G))$ is the colimit of the $\Spc(\cK(E))$ for the elementary abelian \emph{subquotients} of~$G$. More precisely, we introduced a category $\EApp{G}$ of sections $K\normale H\le G$ with $H/K$ elementary abelian and morphisms keeping track of conjugation, restriction \emph{and} quotients.
See details~\cite[Section~11]{balmer-gallauer:TTG-Perm}.

To avoid repeating ourselves unnecessarily, we have decided not to spell out the details in the profinite case.
In this generality, the category~$\EApp{G}$ of \emph{pro-elementary abelian $p$-sections} has objects the pairs $(H,K)$ where $K\le H\le G$ are pro-$p$-subgroups, with $H\le N_G(K)$, and such that $H/K$ is a pro-elementary abelian $p$-group.
Morphisms $(H,K)\to (H',K')$ are defined to be elements $g\in G$ such that~$K'\le K^g\le H^g\le H'$ and composition of morphisms is defined by multiplication in~$G$.
One can then prove that the canonical maps $\Spc(\cK(H/K))\to \Spc(\cK(G))$ assemble into a continuous bijection
\[
  \colim_{E\in\EApp{G}}\SpcKE\to\SpcKG.
\]
However, this map is in general not a homeomorphism, \ie it might not be closed.

This is a shortcoming of the profinite variant.
Another one is that, in the case of $E$ a \emph{finite} elementary abelian group, we produced in~\cite[Section~12]{balmer-gallauer:TTG-Perm} a multi-graded ring~$\Rall(E;\kk)$ and a comparison map realizing $\Spc(\cK(E))$ as an open subspace of the homogeneous spectrum of~$\Rall(E;\kk)$.
And furthermore the ring~$\Rall(E;\kk)$ is a finitely generated $\kk$-algebra.
In the non-finite case, we have an analogue for pro-elementary abelian groups but unfortunately the multi-graded ring becomes a monster, and even the grading involves a non-finitely generated abelian monoid.

In conclusion, the best description of~$\SpcKG$ in the profinite case might well be the formulation as a limit space, as presented in \Cref{Prop:Spc-K(G)-colim}.
\end{Rem}

% ------------------------------------------------------------------------------
\section{Spectrum in the procyclic case}
\label{sec:procyclic}%

%------------------------------------------------------------------------------

%
\begin{Not}
\label{Not:Z_p}
We denote by $\Zp$ the $p$-adic integers viewed as a profinite group.
\end{Not}

\begin{Rem}
\label{Rem:procyclic-Zp}%
A procyclic group~$G$ is abelian and thus the inclusion of its $p$-primary part (its pro-$p$-Sylow) $G_p\into G$ yields a homeomorphism
\[
\rho_{G_p}=\Spc(\Res^G_{G_p}):\Spc(\cK(G_p))\isoto\Spc(\cK(G))
\]
by~\Cref{Exa:pronilpotent}.
Therefore our results on $\Spc(\cK(\Zp))$ in this section easily give similar results for arbitrary procyclic groups.
Cf.~\Cref{Cor:procyclic}.
\end{Rem}

\begin{Rec}
\label{Rec:W*}%
Let $n\ge 1$ and consider the cyclic group~$C_{p^n}$ of order~$p^n$.
We proved in \cite[Proposition~8.3]{balmer-gallauer:TTG-Perm} that the spectrum of $\cK(C_{p^n})$ is the space
\begin{equation}
\label{eq:W*}%
\W^n=\qquad
\vcenter{\xymatrix@R=1em@C=.7em{
{\scriptstyle\gm_0\kern-1em}
& {\bullet} \ar@{-}@[Gray] '[rd] '[rr] '[drrr]
&&
{\bullet}
&{\kern-1em\scriptstyle\gm_1}
&&{\scriptstyle\gm_{n-1}\kern-1em}&\bullet&&\bullet&{\kern-1em\scriptstyle\gm_n}
\\
&{\scriptstyle\gp_1\kern-1em}& {\bullet}&&& {\cdots}& \ar@{-}@[Gray] '[ru] '[rr] '[rrru] &&\bullet&{\kern-1em}\scriptstyle\gp_{n}}}
\end{equation}
consisting of $2n+1$ points, with specialization relations indicated above: Every $\gm_i$ is a closed point and the closure of each $\gp_i$ is~$\{\gm_{i-1},\gp_{i},\gm_i\}$.
A subset is closed if and only if it is specialization-closed.
\end{Rec}

\begin{Def}
\label{Def:W^infty}%
We denote by $\W^{\infty}=\colim_n\W^n$ the colimit along the canonical inclusions $\psi\colon \W^{n}\hook \W^{n+1}$,~$\gp_i\mapsto \gp_i$ and~$\gm_i\mapsto\gm_i$ of the spaces~$\W^n$ of \Cref{Rec:W*}.
\end{Def}

We are now ready to describe~$\Spc(\cK(\Zp))$.
\begin{Thm}
\label{Thm:Zp}%
The spectrum of~$\cK(\Zp)$ is the Alexandroff extension of~$\W^\infty=\cup_{n\ge 0}\W^n$, namely it is the space
$\bar\W^\infty=\SET{\gp_n}{0< n<\infty}\cup\SET{\gm_n}{0\le n\le \infty}=\W^\infty\cup\{\gm_\infty\}$ with specialization relations as follows
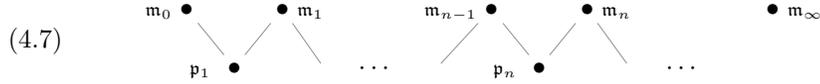
\begin{equation}
\label{eq:Spc(DAM-mod-p)}%
\vcenter{\xymatrix@R=1em@C=.7em{
{\scriptstyle\gm_0\kern-1em}
& {\bullet} \ar@{-}@[Gray] '[rd] '[rr] '[drrr]
&&
{\bullet}
&{\kern-1em\scriptstyle\gm_1}
&&{\scriptstyle\gm_{n-1}\kern-1em}&\bullet&&\bullet&{\kern-1em\scriptstyle\gm_n}&&&\bullet&{\kern-1em\scriptstyle\gm_{\infty}}
\\
&{\scriptstyle\gp_1\kern-1em}& {\bullet}&&& {\cdots}& \ar@{-}@[Gray] '[ru] '[rr] '[rrru] '[rrrr] &{\scriptstyle\gp_{n}}\kern-1em&\bullet &&&{\cdots}}}
\end{equation}
whose closed subsets are those $Z\subseteq\bar\W^\infty$ that are specialization-closed (if $\gp_n\in Z$ for $0<n<\infty$ then $\gm_{n-1},\gm_n\in Z$) and are finite or contain~$\gm_{\infty}$.
The supports of objects in~$\cK(\Zp)$ are exactly the following two classes of closed subsets:
\begin{enumerate}[wide, labelindent=0pt, label=\rm(\Roman*), ref=\rm(\Roman*)]
\smallbreak
\item
\label{it:I}%
The finite specialization-closed subsets that do not contain~$\gm_\infty$.
\smallbreak
\item
\label{it:II}%
The cofinite (finite complement) specialization-closed subsets that contain~$\gm_\infty$.
\end{enumerate}
\end{Thm}

\begin{proof}
We have $\Spc(\cK(\Zp))=\lim_n\Spc(\cK(C_{p^n}))$, by \Cref{Prop:Spc-K(G)-colim}, where the transition maps are induced by inflation. The latter are described in~\cite[Lemma~8.5]{balmer-gallauer:TTG-Perm}. Here is a picture of this system of spaces
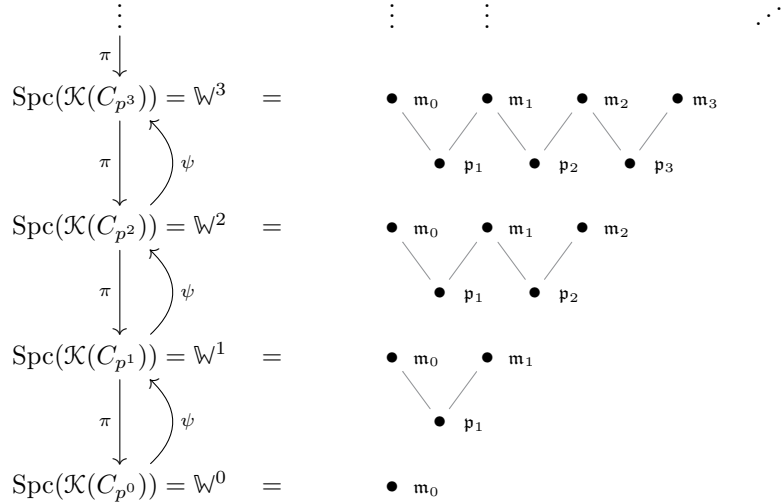
\[
\vcenter{\xymatrix@R=1em@C=.7em{
\vdots  \ar[d]_-{\pi}
&&&& \vdots&& \vdots &&&&&& \iddots
\\
\Spc(\cK(C_{p^3}))=\W^3 \ar[dd]_-{\pi} &=&
&& {\bullet} \ar@{-}@[Gray] '[rd] '[rr] '[rrrd] '[rrrr] '[rrrrrd] '[rrrrrr] &{\kern-1em\scriptstyle\gm_0}
& {\bullet} &{\kern-1em\scriptstyle\gm_1}
& {\bullet} &{\kern-1em\scriptstyle\gm_2}
& {\bullet} &{\kern-1em\scriptstyle\gm_3}
\\
&&&&& {\bullet} &{\kern-1em\scriptstyle\gp_1}
& {\bullet} &{\kern-1em\scriptstyle\gp_2}
& {\bullet} &{\kern-1em\scriptstyle\gp_3}
\\
\Spc(\cK(C_{p^2}))=\W^2 \ar[dd]_-{\pi} \ar@/_2em/[uu]_-{\psi} &=&
&& {\bullet} \ar@{-}@[Gray] '[rd] '[rr] '[rrrd] '[rrrr] &{\kern-1em\scriptstyle\gm_0}
& {\bullet} &{\kern-1em\scriptstyle\gm_1}
& {\bullet} &{\kern-1em\scriptstyle\gm_2}
\\
&&&&& {\bullet} &{\kern-1em\scriptstyle\gp_1}
& {\bullet} &{\kern-1em\scriptstyle\gp_2}
\\
\Spc(\cK(C_{p^1}))=\W^1 \ar[dd]_-{\pi} \ar@/_2em/[uu]_-{\psi} &=&
&& {\bullet} \ar@{-}@[Gray] '[rd] '[rr] &{\kern-1em\scriptstyle\gm_0}
& {\bullet} &{\kern-1em\scriptstyle\gm_1}
\\
&&&&& {\bullet} &{\kern-1em\scriptstyle\gp_1}
\\
\Spc(\cK(C_{p^0}))=\W^0 \ar@/_2em/[uu]_-{\psi} &=&
&& {\bullet} &{\kern-1em\scriptstyle\gm_0}
}}\kern-2em
\]
The transition maps~$\pi\colon \W^{n+1}\onto\W^n$ retract the `obvious' inclusions~$\psi\colon \W^n\hook\W^{n+1}$ and project everything else in~$\W^{n+1}$, that is, $\gp_{n+1}$ and~$\gm_{n+1}$, to the last point~$\gm_n$ in~$\W^n$. So each fiber $\pi\inv(x)$ above $x\in\W^n$ is a single point, unless~$x=\gm_n$ in which case $\pi\inv(x)$ consists of three points.
A point in the limit is given by a coherent sequence of points $(x_n)_{n\ge 0}\in\prod_{n\ge0}\W^n$.
For instance, there is the point $\gm_\infty:=(\gm_n)_{n\ge 0}$ which always picks the right-most point.
Any other point $(x_n)_{n\ge 0}$ must satisfy $x_n\neq \gm_n$ for some~$n$ and by the above comment, has all higher~$x_{m}$ forced to be $\psi^{m-n}(x_n)$. We name those points by their value for~$n$ large enough, yielding the~$\gp_i$ and~$\gm_i$ of the statement. For instance, $\gp_3$ means~$(\ldots,\gp_3,\gp_3,\gp_3,\gp_3,\gm_2,\gm_1,\gm_0)$. Hence $\Spc(\cK(\Zp))=\W^\infty\cup\{\gm_\infty\}$ as a set.

For the topology, every object of~$\cK(\Zp)$ is inflated from some~$c\in\cK(C_{p^n})$ and its support in~$\cK(\Zp)$ is the preimage of $\supp(c)\subseteq\W^n$ under~$\pi\colon \bar\W^\infty\onto\W^n$. If $c$ is acyclic (\ie, its homology vanishes), its support does not touch the right-most point~$\gm_n$ in~$\W^n$, hence its preimage in~$\bar\W^\infty$ will be the (finite) closed $\supp(c)\subseteq \W^n$ included via~$\W^n\hook\W^\infty\hook\bar\W^\infty$, away from~$\gm_{\infty}$. These yield the closed subsets of type~\ref{it:I}. On the other hand, if $c$ \emph{is} supported at $\gm_{n}$ then the preimage of its support in~$\bar\W^\infty$ contains~$\gm_{\infty}$ but also all~$\gp_i$ and~$\gm_i$ for~$i\ge n$. These yield the closed subsets of type~\ref{it:II}.
The rest is general tt-geometry~\cite[Remark~2.7]{balmer:spectrum}: The closed subsets are arbitrary intersections of supports of objects.
The description in the statement is then an exercise.
\end{proof}

We record some easy facts about the topology in $\bar\W^\infty$.
\begin{Rem}
\label{Rem:Zp-not-noeth}%
The spectrum~$\Spc(\cK(\Zp))$ is not noetherian, as the open complement of~$\{\gm_\infty\}$ is not quasi-compact.
\end{Rem}
\begin{Prop}
\label{Prop:V^infty-topology}%
In $\bar\W^\infty$ we have:
\begin{enumerate}[wide, labelindent=0pt, label=\rm(\alph*), ref=\rm(\alph*)]
\item Every $\gm_i$ for $0\le i\le\infty$ is closed.
\smallbreak
\item The closure of $\gp_j$ for $0<j<\infty$ is the set $\{\gm_{j-1},\gp_j,\gm_j\}$.
\smallbreak
\item A subset $V\subseteq\bar\W^\infty$ is quasi-compact if and only if $\gm_\infty\in V$ or $V$ is finite.
\smallbreak
\item A subset $V\subseteq\bar\W^\infty$ is Thomason\,(\,\footnote{\,Recall that a subset is Thomason if it can be written as a union of closed subsets each of whose complement is quasi-compact.}) if and only if (i) $V$ is stable under specializations and (ii) $\gm_\infty\in V$ implies $V$ cofinite.
\qed
\end{enumerate}
\end{Prop}

\begin{Rem}
\label{Rem:Spc-Zp-Thomason}%
Having a list of all Thomason subsets of $\bar\W^\infty\cong\Spc(\cK(\Zp))$, we also get a list of all tt-ideals in $\cK(\Zp)$, by~\cite[Theorem~4.10, Remark~4.3]{balmer:spectrum}.
In particular, there are continuum many distinct ones.
We determine generators for all tt-ideals in \Cref{Rem:Z_p-generators}.
\end{Rem}

\begin{Rem}
\label{Rem:Z_p-rsd}%
Let us briefly discuss `tt-residue fields', \`a la~\cite{balmer-krause-stevenson:ruminations}.
At each of the $\gm_i$ ($0\le i\le\infty$) the residue field is $\Db(k)$, while at each of the $\gp_j$ ($0< j<\infty$) the residue field is $\stab(\kk C_p)$.
For $i<\infty$, the residue field functor $\rsd{\gm_i}$ is given by the left-hand diagram below (cf.~\cite[Proposition~5.15]{balmer-gallauer:TTG-Perm}):
\[
\xymatrix@C=3em{
  \cK(\Zp)
  \ar@[Gray][rd]^-{\rsd{\gm_i}}
  \ar[r]^-{\check{\Psi}^{\ideal{p^i}}}
  \ar[d]_-{\Res^{\Zp}_{\ideal{p^i}}}
  &
  \Db(\kk C_{p^i})
  \ar[d]^-{\Res^{}_1}
  \\
  \cK(\ideal{p^i})
  \ar[r]_-{\check{\Psi}^{\ideal{p^i}}}
  &
  \Db(k).\!\!
}
\qquad\qquad
\xymatrix@C=3em{
  \cK(\Zp)
  \ar@[Gray][rd]^-{\rsd{\gm_\infty}}
  \ar[r]^-{\check{\Psi}^{1}}
  \ar[d]_-{\Res^{\Zp}_{1}}
  &
  \Db(\Zp)
  \ar[d]^-{\Res^{\Zp}_1}
  \\
  \Kb(k)
  \ar[r]_-{\check{\Psi}^{1}=\Id}
  &
  \Db(k).\!\!
}
\]
We write here~$\ideal{p^i}$ for the subgroup of~$\Zp$ of index~$p^i$.
This still makes sense for $i=\infty$ if one interprets $\ideal{p^\infty}$ as the trivial subgroup.
In other words, $\rsd{\gm_\infty}$ is restriction to the trivial subgroup as on the right-hand side above.
These residue field functors on each subcategory~$\cK(G/N)$ were denoted by $\bbF^{\ideal{p^i}}$ in \cite[Definition~7.26]{balmer-gallauer:TTG-Perm}.
Finally, for $0<j<\infty$, the residue field functor $\rsd{\gp_j}$ is given by
\[
\xymatrix@C=3em{
  \cK(\Zp)
  \ar@[Gray]@<-.1em>[rrd]_(.3){\rsd{\gp_j}}
  \ar[r]^-{\check{\Psi}^{\ideal{p^j}}}
  \ar[d]_-{\Res^{\Zp}_{\ideal{p^{j-1}}}}
  &
  \Db(\kk C_{p^j})
  \ar[d]^(.35){\Res^{}_{\ideal{p^{j-1}}}}
  \ar[r]^-{\sta}
  &
  \stab(\kk C_{p^j})
  \ar[d]^-{\Res^{}_{\ideal{p^{j-1}}}}
  \\
  \cK(\ideal{p^{j-1}})
  \ar[r]_-{\check{\Psi}^{\ideal{p^{j}}}}
  &
  \Db(\kk C_p)
  \ar[r]_-{\sta}
  &
  \stab(\kk C_p)
}
\]
where $\sta$ denotes the quotient by perfect complexes, \ie the passage from the derived to the stable module category.
\end{Rem}

\begin{Not}
\label{Not:V^infty-intervals}%
It is convenient to introduce notation for the `half-unbounded intervals' in the space $\bar\W^\infty$.
For any $0\le n\le \infty$:
\begin{itemize}
\item
We denote by $\bar\W^\infty_{\le n}$ the subspace $\{\gm_0,\gp_1,\ldots,\gp_{n},\gm_n\}$.
This can also be identified with $\W^n$ as long as $n<\infty$.
\item
We denote by $\bar\W^\infty_{\ge n}$ the subspace $\{\gm_n,\gp_{n+1},\ldots,\gm_\infty\}$.
\end{itemize}
\end{Not}
\begin{Lem}
\label{Lem:Z_p-generators}%
With notation as above, we have:
\begin{enumerate}[label=\rm(\alph*), ref=\rm(\alph*)]
\item $\supp(\kk(\Zp/\ideal{p^n}))=\bar\W^\infty_{\ge n}$.
\smallbreak
\item $\supp(\kos[\Zp]{\ideal{p^n}})=\bar\W^{\infty}_{\le n-1}$.
\smallbreak
\item $\gm_i=\ideal{\kk(\Zp/\ideal{p^{i+1}}),\kos[\Zp]{\ideal{p^i}}}$ for all $0\le i<\infty$.
\smallbreak
\item $\gm_\infty=\ideal{\kos[\Zp]{\ideal{p^i}}\mid i<\infty}$.
\smallbreak
\item $\gp_j=\ideal{\kk(\Zp/\ideal{p^j}),\kos[\Zp]{\ideal{p^j}}}$ for all $1\le j<\infty$.
\end{enumerate}
\end{Lem}
\begin{proof}
The first two parts are easy consequences of the residue field functors exhibited in \Cref{Rem:Z_p-rsd}.
Or one may simply refer to \cite[Propositions~4.7 and Corollary~7.17]{balmer-gallauer:TTG-Perm}.
The remaining parts then follow by inspecting the supports.
\end{proof}

\begin{Rem}
\label{Rem:Z_p-generators}%
Returning to \Cref{Rem:Spc-Zp-Thomason}, we deduce from \Cref{Lem:Z_p-generators} that each tt-ideal in $\cK(\Zp)$ is generated by a (possibly infinite) family of objects that are either permutation modules~$\kk(\Zp/\ideal{p^n})$ or Koszul complexes~$\kos[\Zp]{\ideal{p^m}}$ or tensor products of such.
To be more explicit, each closed subset of $\bar\W^\infty$ is of the form
\begin{enumerate}[label=\rm(\alph*), ref=\rm(\alph*)]
\item
\label{it:barV^infty-closed-1}%
$\bigcap_{i}(\bar\W^\infty_{\le m_i}\cup \bar\W^\infty_{\ge n_i})$, for $0\le m_i<n_i<\infty$, or
\smallbreak
\item
\label{it:barV^infty-closed-2}%
$Z'\cap \bar\W^\infty_{\le m}$, for $m\ge 0$, and for $Z'$ as in~\ref{it:barV^infty-closed-1}, or
\smallbreak
\item
\label{it:barV^infty-closed-3}%
$Z'\cap \bar\W^\infty_{\ge n}$, for $n\ge 0$, and for $Z'$ as in~\ref{it:barV^infty-closed-1}.
\end{enumerate}
\end{Rem}

\begin{Rem}
\label{Rem:integral-generators}
Each of the tt-ideals in~$\cK(\Zp)=\cK(\Zp;\kk)$ is generated by the images of objects in~$\cK(\Zp;\ZZ)$.
Indeed, the permutation modules $\kk(\Zp/\ideal{p^{i+1}})$ obviously have integral lifts.
The Koszul objects~$\kos[\Zp]{\ideal{p^i}}$ also admit an integral lift if $p$ is odd, see~\cite[Lemma~3.8]{balmer-gallauer:resol-small}.
However, this is not true if~$p=2$.
Still, we know from~\cite[Theorem~3.1]{balmer-gallauer:resol-small} that there is an acyclic complex~$D\in\cK(C_{2^i};\ZZ)$ concentrated in non-negative degrees, with~$D_0=\ZZ$ and $D_1$ a free $\ZZ C_{2^i}$-module.
It then follows from~\cite[Corollary~3.20]{balmer-gallauer:TTG-Perm} that~$D$ generates the ideal of acyclics.
The same remains true for its image in~$\cK(C_{2^i};\kk)$.
Inflating these complexes to~$\cK(\Zp;\ZZ)$ yields the claim.\,(\footnote{\,This argument should have been spelled out in the proof of~\cite[Theorem~11.3]{balmer-gallauer:rage}.})
\end{Rem}
%------------------------------------------------------------------------------

\section{About stratification}
\label{sec:stratification}%

%------------------------------------------------------------------------------

We isolate a couple of abstract tt-geometric results related to stratification.
Both results are close to folklore but we did not find a convenient reference.
We recall that `big' tt-categories mean rigidly-compactly generated ones and a `geometric' tt-functor means one that preserves coproducts (and compact objects since they are assumed to agree with rigid objects).
\begin{Lem}
\label{Lem:fully-faithful-localization}%
Let $F\colon\cT\to\cS$ be a geometric tt-functor.
Let $V\subseteq\Spc(\cT^c)$ be a Thomason subset with preimage $W:=\Spc(F)\inv(V)$ in~$\Spc(\cS^c)$ and let
\[
\bar F\colon\frac{\cT}{\Loc{\cT^c_V}}\to\frac{\cS}{\Loc{\cS^c_W}}.
\]
be the functor induced by~$F$. If $F$ is fully faithful then so is $\bar F$.
\end{Lem}
\begin{proof}
Recall that $\cT^c_V$ and $\cS^c_W$ are the tt-ideals of~$\cT^c$ and~$\cS^c$ supported on~$V$ and~$W$, respectively (and $W$ is Thomason because $\Spc(F)$ is spectral).
By \cite[Theorem~6.3]{balmer-favi:idempotents}, we have the following equality of tt-ideals in~$\cS^c$:
\begin{equation}
\label{eq:aux-ff}%
\cS^c_W=\ideal{\SET{F(c)}{c\in\cT^c_V}}.
\end{equation}
Since $F$ is coproduct-preserving, it maps $\Loc{\cT^c_V}$ into $\Loc{\cS^c_W}$, hence the functor~$\bar{F}$.
Let now $U\colon\cS\to\cT$ be the right adjoint to~$F$.
We claim that $U$ sends $\Loc{\cS^c_W}$ into $\Loc{\cT^c_V}$.
As $F$ preserves compacts, $U$ is coproduct-preserving and it suffices to show $U(\cS^c_W)\subseteq\Loc{\cT^c_V}$.
By~\eqref{eq:aux-ff} the tt-ideal $\cS^c_W$ is generated as a thick subcategory of~$\cS^c$ by objects of the form $F(c)\otimes d$, where $c\in\cT^c_V$ and $d\in\cS^c$.
By the projection formula, $U$ sends such objects to
\[
U(F(c)\otimes d)\cong c\otimes U(d)\in\Loctens{\cT^c_V}=\Loc{\cT^c_V}
\]
which proves the claim.
Consequently, the adjoint $U$ also descends to a functor
\[
\bar U\colon\frac{\cS}{\Loc{\cS^c_W}}\to\frac{\cT}{\Loc{\cT^c_V}}
\]
which is automatically right adjoint to~$\bar F$.
The unit of this adjunction $\bar{F}\adj\bar{U}$ is given by the unit of the original adjunction $F\dashv U$ viewed in the localization.
As the latter is invertible, by assumption, so is the former.
\end{proof}

\begin{Rec}
\label{Rec:BHS-stratification}%
When $\cT$ is a big tt-category whose spectrum $\Spc(\cT^c)$ is weakly noetherian, we have idempotents~$g(\cP)\in\cT$ for every $\cP\in\Spc(\cT^c)$ and a support theory for big objects given by $\Supp(t)=\SET{\cP}{g(\cP)\otimes t\neq 0}\subseteq\Spc(\cT^c)$. See~\cite[Definition~7.16]{balmer-favi:idempotents}.
We say that $\cT$ is \emph{BHS-stratified} if $\Spc(\cT^c)$ is weakly noetherian and if the localizing $\otimes$-ideals $\cL$ of~$\cT$ are in bijection with the subsets $X\subseteq\Spc(\cT^c)$ of the spectrum, via $\cL\mapsto \Supp(\cL)=\cup_{t\in \cL}\Supp(t)$ and $X\mapsto \SET{t\in\cT}{\Supp(t)\subseteq X}$.
See~\cite[Definition~4.4]{barthel-heard-sanders:stratification-Mackey}.
\end{Rec}

The question of descending stratification along a `sufficiently conservative' tt-functor has been discussed in~\cite{barthel-heard-sanders:stratification-Mackey,barthel-et-al:descent} for instance. These methods yield the following result, which is similar to~\cite[Theorem~17.20]{barthel-et-al:cosupport}.
\begin{Prop}
\label{Prop:descent-stratification}%
Let $F\colon\cT\to\cS$ be a geometric functor with right adjoint \mbox{$U\colon\cS\to\cT$} satisfying the following conditions:
\begin{enumerate}[label=\rm(\arabic*), ref=\rm(\arabic*)]
\item
The map $\varphi:=\Spc(F)\colon \Spc(\cS^c)\to \Spc(\cT^c)$ is closed and injective.
\item
The tt-category~$\cS$ is BHS-stratified.
\item
\label{it:sf-3}%
The unit $\unit$ belongs to $\Loctens{U(\unit)}$ in~$\cT$.
\end{enumerate}
Then $\cT$ is BHS-stratified.
\end{Prop}
\begin{proof}
It follows easily from~\ref{it:sf-3} and the projection formula that $F$ is conservative, hence $\varphi$ is actually surjective~\cite[Proposition~7.1]{balmer-gallauer:TTG-Perm}, \ie a homeomorphism. Thus $\Spc(\cT^c)$ is weakly noetherian as well.
Let $\cP\in\Spc(\cT^c)$ and let $\cQ\in\Spc(\cS^c)$ be the unique preimage, \ie $\varphi(\cQ)=\cP$.
Then, by~\cite[Theorem~6.3]{balmer-favi:idempotents}, $F(g(\cP))=g(\varphi\inv(\{\cP\}))=g(\cQ)$.
For every $t\in\cT$ it follows that $t\otimes g(\cP)$ vanishes in~$\cT$ if and only if $F(t\otimes g(\cP))\simeq F(t)\otimes g(\cQ)$ vanishes in~$\cS$. In other words, $\Supp(F(t))=\varphi\inv(\Supp(t))$ for all~$t\in\cT$.
The result now follows from \cite[Proposition~12.7]{barthel-et-al:descent}.
\end{proof}

Here is an application to our categories~$\cT(G)=\DPerm(G;\kk)$ and $\cT(G)^c=\cK(G)$.

\begin{Cor}
\label{Cor:stratification-along-Res}%
Under \Cref{Hyp:general}, let $H\le G$ be a closed subgroup of index prime to~$p$ (\Cref{Def:index-prime-to-p}).
Assume that the tt-category $\cT(H)$ is BHS-stratified (\Cref{Rec:BHS-stratification}).
If $\rho_H=\Spc(\Res^G_H):\Spc(\cK(H))\to\Spc(\cK(G))$ is injective then $\rho_H$ is a homeomorphism and the tt-category $\cT(G)$ is BHS-stratified as well.
\end{Cor}
\begin{proof}
The map $\rho_H$ is closed by \Cref{Prop:rho-closed}.
Thus we apply \Cref{Prop:descent-stratification} to $F=\Res^G_H$. Hypothesis~\ref{it:sf-3} is satisfied in this case by \Cref{Prop:Res-split-faithful}.
\end{proof}

%------------------------------------------------------------------------------

\section{Stratification in the procyclic case}
\label{sec:stratification-procyclic}%

%------------------------------------------------------------------------------

Having established a fairly complete picture of the tt-geometry of $\cK(\Zp)$ in \Cref{sec:procyclic}, we now turn to big tt-categories.
Our goal is to show BHS-stratification for the big tt-category~$\cT(\Zp)$; see~\Cref{Rec:BHS-stratification}.
As mentioned already, the space $\bar\W^\infty\cong\Spc(\cK(\Zp))$ is not noetherian which complicates matters a bit.
It is, however, \emph{generically noetherian} in the sense of~\cite[Definition~9.5]{barthel-heard-sanders:stratification-Mackey}.
\begin{Lem}
\label{Lem:Spc-Z_p-generically-noetherian}%
The space $\bar\W^\infty$ is generically noetherian, that is, the generalization closure of every point is noetherian.
\end{Lem}
\begin{proof}
By \Cref{Prop:V^infty-topology} the generalization closures of all the points are finite: They are $\{\gm_0,\gp_1\}$, $\{\gp_n\}$ and $\{\gp_{n-1},\gm_n,\gp_{n}\}$ for $1\le n<\infty$ and $\{\gm_\infty\}$.
\end{proof}

\begin{Rem}
\label{Rem:Spc-Z_p-weakly-noetherian}%
Every generically noetherian space~$X$ is weakly noetherian, that is, every point~$x\in X$ can be written as the intersection $\{x\}=V\cap W^c$ of a Thomason subset~$V$ with the complement of a Thomason subset~$W$, see~\cite[Lemma~9.9]{barthel-heard-sanders:stratification-Mackey}.
In our case, for each $x\in\bar\W^\infty$, we can realize $\{x\}=V\cap W^c$ as follows:
\begin{equation}\label{eq:weakly-visible}%
\begin{tabular}{cc|ccc|ccc|cc}
  $x$&&&$V$&&&$W^c$&&&$W$
  \\[.1em]
  \hline
  $\gm_\infty$&&&$\bar\W^\infty$&&&$\{\gm_{\infty}\}$&&&$\W^\infty\vphantom{I^{I^I}}$\\[.2em]
  $\gm_i$&&&$\{\gm_i\}$&&&$\bar\W^\infty$&&&$\emptyset$\\[.2em]
%  $\gm_i$&&&$\bar\W^\infty$&&&$\bar\W^\infty\sminus\{\gm_i\}$\\
  $\gp_j$&&&$\bar\W^\infty$&&&$\{\gp_j\}$&&&$\bar\W^\infty\sminus\{\gp_j\}$
\end{tabular}
\end{equation}
\end{Rem}

\begin{Lem}
\label{Lem:Z_p-LTG}
The big tt-category $\cT(\Zp)$ satisfies the local-to-global principle in the sense of~\cite[Definition~3.8]{barthel-heard-sanders:stratification-Mackey}.
\end{Lem}
\begin{proof}
Compare the proof of~\cite[Theorem~3.6]{BIK:stratifying-stmod-kG} and its translation in~\cite[Theorem~3.21]{barthel-heard-sanders:stratification-Mackey}.
Consider the localizing ideal of~$\cT(\Zp)$
\[
\cL:=\Loctens{g(\cP)\mid \cP\in\Spc(\cK(\Zp))}
\]
where $g(\cP)$ is as in \Cref{Rec:BHS-stratification}.
We need to prove that~$\cL$ is the whole of~$\cT(\Zp)$.
Recall that if $\{\cP\}$ is Thomason then $g(\cP)=e_{\{\cP\}}$ and if $\{\cP\}^c$ is Thomason then $g(\cP)=f_{\{\cP\}^c}$.
The first case occurs for all $\gm_i$, $0\le i<\infty$ while the second case occurs for $\gm_\infty$ as well as all the $\gp_j$.
See~\eqref{eq:weakly-visible}.
We deduce that all of the following idempotents belong to $\cL$:
\[
e_{i}:= e_{\{\gm_i\}}, \,0\le i<\infty, \quad f_j:=f_{\{p_j\}^c}, \,0<j<\infty, \quad f_\infty:= f_{\W^\infty}.
\]
We will now prove the following two claims:
\begin{enumerate}[label=\rm(\arabic*)]
\item For all $0\le i< \infty$, the idempotent $e_{\le i}:=e_{\W^{i}}$ also belongs to~$\cL$.
\item We have $f_\infty=\hocolim_i f_{\le i}$ (where $f_{\le i}:=f_{\W^i}$).
\end{enumerate}
Together these would imply the lemma.
Indeed, by the first claim $\hocolim_i e_{\le i}$ belongs to~$\cL$.
By the second claim and the exact triangle
\[
\hocolim_i e_{\le i}\to\unit\to\hocolim_i f_{\le i}
\]
we then conclude that $\unit\in\cL$ as required.

The second claim follows from generalities on idempotents and the fact that $\W^\infty=\cup_{i}\W^i$.
For the first claim we proceed by induction on $i$.
The case $i=0$ is clear since $e_{\le 0}=e_0\in\cL$.
And assuming $e_{\le i}\in\cL$ we have also $e_{\le i}\vee e_{i+1}\in\cL$ by the Mayer-Vietoris triangle~\cite[Theorem~3.13]{balmer-favi:idempotents}.
Consider then the triangle
\[
e_{\le i}\vee e_{i+1}\to e_{\le i+1}\to e_{\le i+1}\otimes f_{\W^{i}\cup\{\gm_{i+1}\}}.
\]
It is easy to see that the last term is equal to $e_{\le i+1}\otimes f_{i+1}\in\cL$ thus the claim.
\end{proof}

For the next result recall Krause's homotopy category of injectives~$\K\Inj(G;\kk)$ from~\cite[Section~2]{krause:stabX}.
It is a compactly generated tensor triangulated category whose compact part identifies with $\Db(\kkG)$.
\begin{Prop}
\label{Prop:hereditary-minimality}%
Let $G$ be a profinite group with $p$-cohomological dimension~$\le 1$.
Then $\K\Inj(G;\kk)$ has a unique non-trivial localizing ideal and $\Spc(\Db(\kkG))=\ast$.
\end{Prop}
\begin{proof}
Our assumption buys us that $\Hm^i(G;N)=0$ for all $N\in\Mod(G;\kk)$ and $i>1$.
Let $M\in\Mod(G;\kk)$ be a finite-dimensional discrete module.
Note that for each $i>1$ we have
\begin{align*}
               \Ext^i(M,N)
             \cong\Ext^i(\kk,M^*\otimes N)\cong\Hm^i(G;M^*\otimes N) = 0.
\end{align*}
We would like to deduce the same for an infinite-dimensional~$M$ and for this, naturally, we write $M=\colim_\alpha M_\alpha$ as a union of finite-dimensional ones.
It then suffices to show that the canonical map
\begin{equation}
\label{eq:ext-continuous}
\Ext^i(M,N)=\Ext^i(\colim M_\alpha,N)\to\lim\Ext^i(M_\alpha,N)=0
\end{equation}
is an isomorphism.
Choose an injective resolution $N\to I^\bullet$ and let $C_\alpha=\Hom(M_\alpha,I^\bullet)$ be the cochain complex computing the relevant Ext-groups.
An inclusion $M_\alpha\into M_\beta$ induces a degreewise surjection $C_\alpha\twoheadleftarrow C_\beta$ as $I^\bullet$ is degreewise injective.
It then follows~\cite[Theorem~3.5.8]{weibel:homalg} that~\eqref{eq:ext-continuous} is surjective with kernel
\[
{\lim}^1\Ext^{i-1}(M_\alpha,N)\cong{\lim}^1\Hm^{i-1}(G;M_\alpha^*\otimes N).
\]
If $i>2$ these terms clearly vanish.
And for $i=2$ we can use that the functor $\Hm^1(G;-)$ is right exact, again by our assumption on~$G$, hence the transition maps in the inverse system are surjective and the higher inverse limit vanishes~\cite[Proposition~3.5.7]{weibel:homalg}.

In other words, the abelian category $\Mod(G;\kk)$ is hereditary, which has two consequences that are relevant for us.
First, we have $\K\Inj(G;\kk)\simeq\D(\kkG)$, the unbounded derived category of~$\Mod(G;\kk)$ by~\cite[Example~3.10]{krause:stabX}.
And secondly, in $\D(\kkG)$ every object $t$ splits as $t\cong\oplus_{i\in\ZZ}\Hm_i(t)[i]$ so that $\Hom(t,s)$ is in bijection with families $\alpha_i:\Hm_i(t)\to \Hm_i(s)$ and $\beta_i\in\Ext^1_{\Mod(G;\kk)}(\Hm_i(t),\Hm_{i+1}(s))$.
In particular, if $t$ and $s$ are both non-zero then
\[
\Hom_{\D(\kkG)}(t,s[m])\neq 0
\]
for a suitable choice of~$m\in\ZZ$.
By~\cite[Lemma~3.9]{BIK:stratifying-stmod-kG} we conclude that $\D(\kkG)\simeq\K\Inj(G;\kk)$ is the minimal non-zero localizing tensor ideal as announced.
\end{proof}

\begin{Thm}
\label{Thm:Z_p-stratification}%
The big tt-category $\cT(\Zp)$ is BHS-stratified (\Cref{Rec:BHS-stratification}).
Moreover, it satisfies the telescope conjecture.
\end{Thm}
\begin{proof}
We verified that $\Spc(\cK(\Zp))$ is weakly noetherian in \Cref{Rem:Spc-Z_p-weakly-noetherian}.
We also verified that $\cT(\Zp)$ satisfies the local-to-global principle in \Cref{Lem:Z_p-LTG}.
It remains to prove the minimality property (\cite[Theorem~4.1]{barthel-heard-sanders:stratification-Mackey}).
By~\cite[Corollary~5.3]{barthel-heard-sanders:stratification-Mackey}, we need to show for every $\cP\in\Spc(\cK(\Zp))$ that the local category $\cT(\Zp)/\Loc{\cP}$ satisfies minimality at its closed point.
For $\cP\in\W^\infty$, there exists some $0\le n<\infty$ such that $\cP\in\bar\W^\infty_{\le n}$.
Consider the inflation functor $F\colon\cT(C_{p^{n+1}})\to\cT(\Zp)$ and let $\cQ:=F\inv(\cP)\in\W_{n+1}$ be the corresponding prime.
Let $V\subseteq\W_{n+1}$ be the Thomason subset $\supp(\cQ)$ and $W\subseteq\bar\W^\infty$ its preimage under~$\Spc(F)$.
Now it pays off that we inflated from $C_{p^{n+1}}$ (instead of $C_{p^n}$, for example): This ensures that $\Spc(F)$ restricts to a bijection between the generalizations of~$\cP$ and those of $\cQ$ so that we have $W=\supp(\cP)$.
(Indeed, the support of a prime is the complement of its generalization-closure.)
By \Cref{Lem:fully-faithful-localization}, the functor $F$ then induces a fully faithful tt-functor
\begin{equation}
\label{eq:bar-inflation}
\frac{\cT(C_{p^{n+1}})}{\Loc{\cQ}}\longrightarrow\frac{\cT(\Zp)}{\Loc{\cP}}
\end{equation}
which turns out to be an equivalence.
Indeed, it suffices to show that every $\kk(\Zp/\ideal{p^m})$ is in the essential image.
For $m> n+1$ these are zero, by \Cref{Lem:Z_p-generators}.
And the remaining ones are indeed inflated from $C_{p^{n+1}}$.
Minimality of $\cT(\Zp)/\Loc{\cP}$ at its closed point now follows from this equivalence~\eqref{eq:bar-inflation} and the reverse direction of~\cite[Corollary~5.3]{barthel-heard-sanders:stratification-Mackey} since we know stratification for finite groups, by \cite[Theorem~9.11]{balmer-gallauer:TTG-Perm}.

For $\cP=\gm_\infty$ we need to show that $\cT(\Zp)/\Loc{\gm_\infty}=\K\Inj(\Zp;\kk)$
satisfies minimality.
(For the identification of the two categories see~\cite[Remark~4.21]{balmer-gallauer:resol-big}.)
But $\Zp$ is a free pro-$p$-group and therefore has $p$-cohomological dimension~one~\cite[Corollary~7.5.2]{ribes-zalesskii:profinite}.
Minimality then follows from \Cref{Prop:hereditary-minimality}.

The telescope conjecture then follows from \Cref{Lem:Spc-Z_p-generically-noetherian} together with~\cite[Theorem~9.11]{barthel-heard-sanders:stratification-Mackey}.
\end{proof}

\begin{Cor}
\label{Cor:procyclic}%
Let $G$ be a profinite group of the form~$G=H\rtimes \Zp$ where $H$ is a group of order prime to~$p$. (For instance, $G$ could be an abelian or more generally a pro-nilpotent profinite group, with $p$-Sylow $\Zp$, \eg $G=\hat{\bbZ}$ the profinite integers.)
Then $\Spc(\cK(G))\cong\bar\W^\infty$ and $\cT(G)$ is BHS-stratified and satisfies the telescope conjecture.
\end{Cor}
\begin{proof}
By \Cref{Rem:procyclic-Zp}, the map $\rho_{\Zp}:\Spc(\cK(\Zp))\to\SpcKG$ is a homeomorphism.
Stratification then follows from \Cref{Thm:Z_p-stratification,Cor:stratification-along-Res}.
And the telescope conjecture follows again from \Cref{Lem:Spc-Z_p-generically-noetherian} together with~\cite[Theorem~9.11]{barthel-heard-sanders:stratification-Mackey}.
\end{proof}

\begin{Exa}
\label{Exa:non-weakly-noetherian}%
We finish with an example showing that the results in this section will not generalize to arbitrary profinite groups.

Let $G=(C_p)^{\bbN}$ be a countably infinite pro-elementary abelian group.
In that case, the associated cohomological support space is
\[
\Vee{G}\cong\Spech(\Hm^\sbull(G;\kk))\cong\Spech\left(\kk[\xi_i\mid i\in\bbN]\otimes_k\Lambda_{\kk}(\eta_i\mid i\in\bbN)\right)\cong\bar{\bbP}^{\infty}_{\kk},
\]
an infinite-dimensional projective space with a unique closed point attached on top.
We claim that this space is not weakly noetherian.
More precisely, that the unique closed point is not visible (equivalently, not weakly visible).
Recall this amounts to showing that it is not the support of any object in~$\Db(\kk G)$.
But every such object is inflated from a finite quotient~$(C_p)^I$, where $I\subset\bbN$ is a finite subset.
Let~$J=\bbN\bs I$ be the complement and $t\in\Db(\kk (C_p)^I)$ a non-zero object.
Writing $\pi=\Spc(\Infl)$ and $\rho=\Spc(\Res)$, with appropriate decorations, we then have
\begin{align*}
  \rho_J\inv(\supp(\Infl^{C_p^I}_{G}(t)))&=\rho_J\inv\pi_I\inv(\supp(t))\\
                                         &=\pi_{\emptyset}\inv\rho_{\emptyset}\inv(\supp(t))\\
                                         &=\pi_{\emptyset}\inv(\Spc(\Db(\kk)))\\
                                         &=\Spc(\Db(\kk(C_p^J)))
\end{align*}
using the fact that the composite $(C_p)^J\to G\to (C_p)^I$ factors through the trivial group~$\ast=(C_p)^\emptyset$.
This shows that the support of~$\Infl^{C_p^I}_{G}(t)$ contains much more than just the unique closed point in~$\bar{\bbP}^\infty$.

As $\check{\psi}^1\colon\Spc(\Db(\kk G)) \hook \SpcKG$ is a spectral map, we deduce that the latter is not weakly noetherian either (nor, \textsl{a fortiori}, generically noetherian).
In anticipation of the next section we also mention that restricting to absolute Galois groups does not change these comments.
For example, the absolute Galois group~$\Zp^{\bbN}$ admits~$G$ as a quotient.
(Here, we are using~\cite[p.~352]{geyer:abelian-absolute-galois}.)
Letting~$H\le\Zp^{\bbN}$ be the kernel of this quotient map we observe, similarly, that $\psi^H\colon\SpcKG\hook\Spc(\cK(\Zp^{\bbN}))$ is spectral so that $\Spc(\cK(\Zp^{\bbN}))$ cannot be weakly noetherian either.
\end{Exa}

%------------------------------------------------------------------------------

\section{Artin motives over a finite base field}
\label{sec:artin-finite-fields}%

%------------------------------------------------------------------------------

We can translate from procyclic groups to Artin motives over finite fields, via Voevodsky's result~\cite[\S\,3.4]{Voevodsky00}.

\begin{Rec}
\label{Rec:Artin-motives}%
Let $\FF$ be a field and denote by $G=\Gal(\FF^{\mathrm{sep}}/\FF)$ its absolute Galois group with the Krull topology.
The classical Grothendieck-Galois correspondence induces a canonical equivalence of tt-categories $\DPerm(G;\kk)\simeq \DAMbig(\FF;\kk)$ with the triangulated \emph{category of Artin motives}.
In particular, the \emph{geometric Artin motives} $\DAM(\FF;\kk)$, which form the compact part, identify with $\cK(G)$.
We refer for a detailed account of these equivalences (and more) to the survey article~\cite{balmer-gallauer:Dperm}.
\end{Rec}

\begin{Rem}
\label{Rem:wild}%
It follows from Voevodsky's result and the usual `wilderness' of modular representation theory that $\DAM(\bbF;\kk)$ is typically wild
when $\chara(\kk)=p>0$ and the absolute Galois group admits quotients of large enough $p$-rank.
\end{Rem}

\begin{Cor}
\label{Cor:DAM}%
Let $\FF$ be a finite field, or any field with abelian (or pro-nilpotent) absolute Galois group whose $p$-Sylow is isomorphic to~$\Zp$.
\begin{enumerate}[wide,labelindent=0pt,label=\rm(\alph*), ref=\rm(\alph*)]
\item
\label{it:DAM-Spc}%
The space $\Spc(\DAM(\FF;\kk))$ canonically identifies with $\bar\W^\infty$ of~\Cref{Thm:Zp}.
\smallbreak
\item
\label{it:DAM-stratified}%
The big tt-category $\DAMbig(\FF;\kk)$ is stratified and satisfies the telescope conjecture.
\end{enumerate}
\end{Cor}
\begin{proof}
By \Cref{Rec:Artin-motives}, we have $\DAMbig(\FF;\kk)\simeq\cT(\hat\bbZ)$ and use \Cref{Cor:procyclic}.
\end{proof}

In the remainder of the section we compute the spectrum of~$\DAM(\FF;\ZZ)$ with integral coefficients.
\begin{Cor}
\label{Cor:Fq-stratification-over-Z}%
Let $\comp\colon\Spc(\DAM(\FF;\ZZ))\to\Spec(\ZZ)$ be the comparison map.\,(\,\footnote{\,Recall that it sends a prime~$\cP$ to the (necessarily) prime ideal of integers~$n$ such that the cone of~$\ZZ\xto{n}\ZZ$ does not belong to~$\cP$.})
\begin{enumerate}[wide, labelindent=0pt, label=\rm(\alph*), ref=\rm(\alph*)]
\item Its fiber over~$(0)$ is a singleton:~$\comp\inv(\{0\})=:\{\gm_{\infty}(0)\}$;
\item Its fiber over~$(p)$ with~$p>0$ is a copy of the space~$\bar\W^\infty$, denoted~$\bar\W^\infty(p)$.
\end{enumerate}
The supports of objects are precisely those subsets that intersect each $\bar\W^\infty(p)$ in a closed subset of~$\bar\W^\infty$ and in addition:
\begin{enumerate}[label=\rm(\Roman*),ref=\rm(\Roman*)]
\item \label{it:type-finite} either, do not contain $\gm_\infty(0)$ and meet only finitely many~$\bar\W^\infty(p)$;
\item \label{it:type-cofinite} or, contain $\gm_\infty(0)$ and are cofinite.
\end{enumerate}
In particular, these subsets form a basis of closed subsets in~$\Spc(\DAM(\FF;\ZZ))$.
\end{Cor}

\begin{Rem}
Of course, the condition of a subset of~$\bar\W^\infty$ being closed in the statement can be made explicit using \Cref{Thm:Zp,Prop:V^infty-topology}.
\end{Rem}

\begin{proof}[Proof of \Cref{Cor:Fq-stratification-over-Z}]
Let $G$ be the absolute Galois group of~$\FF$.
We translate the questions to modular representation theory for~$G$ as usual.
The fiber over~$(0)$ identifies with the spectrum of~$\cK(G;\ZZ)\otimes\QQ\simeq\cK(G;\QQ)$.
This is a singleton
\[
\gm_\infty(0)=\ker(\cK(G;\ZZ)\to\cK(G;\QQ))
\]
by \Cref{Rem:Spc-to-Sylow}.

Now, let $p>0$ and denote by $L_p\colon \cK(G;\ZZ)\to\cK(G;\ZZ/p)$ the base change functor, with right adjoint~$R_p$.
The fiber of the comparison map over~$(p)$ is the support of~$\cone(\ZZ\xto{p}\ZZ)=R_p(\ZZ/p)$ hence
\[
L_p^*\colon \Spc(\cK(G;\ZZ/p))\to\Spc(\cK(G;\ZZ))
\]
surjects onto this fiber, by~\cite[Theorem~1.7]{balmer:surjectivity}.
We claim that this map is in fact a homeomorphism onto its image.
For injectivity, it suffices to show that every prime ideal in~$\cK(G;\ZZ/p)$ is generated by the image of~$L_p$, see~\cite[Proposition~2.10(a)]{balmer-gallauer:rage}.
As~$L_p$ commutes with inflation along~$G\onto\Zp$ we reduce to~$G=\Zp$ for this question.
We may then apply \Cref{Rem:integral-generators}.
To prove that~$L_p^*$ is a closed map, it suffices to show it has the going-up property~\cite[Theorem~5.3.3]{dst:spectral-spaces}.
So let~$x\in\Spc(\cK(G;\ZZ/p))$ and $L_p^*(x)\rightsquigarrow y$.
Since the image of~$L_p^*$ is closed, $y=L_p^*(x')$ for some~$x'$.
It then follows again from~\cite[Proposition~2.10(a)]{balmer-gallauer:rage} together with \Cref{Rem:integral-generators} that $x\rightsquigarrow x'$.

For the second statement let $t\in\cK(G;\ZZ)$ and set $Z=\supp(t)$.
As $L_p^*$ is continuous we deduce that $Z\cap\bar\W^\infty(p)$ is closed in $\bar\W^\infty$ for every $p>0$, as claimed.
Assume first that $t\in\gm_\infty(0)$.
By~\cite[Proposition~5.10]{gallauer:tt-dtm-algclosed}, $\supp(t)\cap\bar\W^\infty(p)=\emptyset$ for almost all~$p$.
This shows that $Z$ is of type~\ref{it:type-finite}.
Conversely, given a subset~$T$ of type~\ref{it:type-finite} we may without loss of generality assume it is contained in a single~$\bar\W^\infty(p)$.
We observed already above (see \Cref{Rem:Z_p-generators,Rem:integral-generators}) that there is an object $t\in\cK(G;\ZZ)$ such that $\supp(t)\cap\bar\W^\infty(p)=T$.
But then $\supp(t\otimes \cone(\ZZ\xto{p}\ZZ))=T$.

Now assume $t\notin\gm_\infty(0)$ and fix a prime~$p>0$ with associated base change functor $L_p$ as above.
As $\Res^G_1\colon\cK(G;\QQ)\to\cK(1;\QQ)$ is conservative, we deduce that
$\Res^G_1(t)\otimes\QQ=\Res^G_1(t\otimes\QQ)\neq 0$.
But then, $\Res^G_1(t)\in\cK(1;\ZZ)$ is an object that cannot vanish after applying~$L_p$ either.
In other words, $\Res_1^G(L_p(t))=L_p(\Res^G_1(t))\neq 0$ in~$\cK(1;\ZZ/p)$ and this means that $L_p(t)\notin\gm_\infty$, or $\gm_\infty\in \supp(L_p(t))$.
It follows from \Cref{Thm:Zp} that $Z\cap\bar\W^\infty(p)$ is cofinite.
To show that $Z$ is of type~\ref{it:type-cofinite} it remains to show that it contains all but finitely many $\bar\W^\infty(p)$.
But $t$ is inflated from $G/N$ for some $N\normaleop G$.
All but finitely many~$p$ are coprime to the index~$[G:N]$ and for such~$p$, $L_p(t)$ is inflated from an object $t'\in\cK(G/N;\ZZ/p)$.
As $L_p(t)\neq 0$ we have $t'\neq 0$ so that $t'$ generates the entire tt-category~$\cK(G/N;\ZZ/p)$ by \Cref{Rem:Spc-to-Sylow}.
We deduce that $L_p(t)$ generates $\cK(G;\ZZ/p)$ showing that $\bar\W^\infty(p)\subseteq Z$.

Conversely, let $T$ be a subset of type~\ref{it:type-cofinite}.
Without loss of generality $T\supset\bar\W^\infty(p)$ for all but one prime~$p$.
Let $t'\in\cK(G;\ZZ)$ be an object such that $\supp(t')\cap\bar\W^\infty(p)=T\cap\bar\W^\infty(p)$ for that particular~$p$.
As explained above, such a~$t'$ exists and we may assume it is inflated from $G/N'$ for some $N'\normaleop G$ of index a $p$-power.
Let $N<N'$ be of index~$p$ and consider $t:=t'\oplus \ZZ(G/N)$.
Note that $L_p(t)$ and $L_p(t')$ have the same support $(L_p^*)\inv(T)$.
For every prime~$q\neq p$, $L_q(t)=L_q(t')\oplus \ZZ/q(G/N)$ is inflated from $\cK(G/N;\ZZ/q)$ and not acyclic hence generates the tt-category (\Cref{Rem:Spc-to-Sylow}).
Clearly, $t\notin\gm_\infty(0)$ as well and this shows that $\supp(t)=T$.

The last statement is general tt-geometry: the supports of objects form a basis for the topology.
\end{proof}

%------------------------------------------------------------------------------

% \bibliographystyle{alpha}
% \bibliography{ref}

\newcommand{\etalchar}[1]{$^{#1}$}

%------------------------------------------------------------------------------

\end{document}